\documentclass[12pt,oneside,english]{amsart}
\usepackage[latin9]{inputenc}
\usepackage{amsthm}
\usepackage{amstext}
\usepackage{amssymb}

\makeatletter

\newcommand{\noun}[1]{\textsc{#1}}

\numberwithin{equation}{section}
\numberwithin{figure}{section}
\theoremstyle{plain}
\ifx\thechapter\undefined
\newtheorem{thm}{\protect\theoremname}
\else
\newtheorem{thm}{\protect\theoremname}[chapter]
\fi
  \theoremstyle{definition}
  \newtheorem{defn}[thm]{\protect\definitionname}
  \theoremstyle{remark}
  \newtheorem{rem}[thm]{\protect\remarkname}
  \theoremstyle{remark}
  
  \theoremstyle{plain}
  \newtheorem{prop}[thm]{\protect\propositionname}
  \theoremstyle{plain}
  \newtheorem{fact}[thm]{\protect\factname}
  \theoremstyle{plain}
  \newtheorem{lem}[thm]{\protect\lemmaname}
  \theoremstyle{plain}
  \newtheorem{cor}[thm]{\protect\corollaryname}
  \theoremstyle{remark}
  \newtheorem*{claim*}{\protect\claimname}

  \theoremstyle{remark}

\makeatother

\usepackage{babel}
  \providecommand{\claimname}{Claim}
  \providecommand{\corollaryname}{Corollary}
  \providecommand{\definitionname}{Definition}
  \providecommand{\factname}{Fact}
  \providecommand{\lemmaname}{Lemma}
  \providecommand{\propositionname}{Proposition}
  \providecommand{\remarkname}{Remark}
  \providecommand{\examplename}{Example}
\providecommand{\theoremname}{Theorem}
\providecommand{\examplename}{Example}

\begin{document}
\global\long\def\Rr{\mathbb{R}}
 \global\long\def\dim#1{\mathrm{d}\left(#1\right)}
 \global\long\def\cl#1{\mathrm{cl}\left(#1\right)}
 \global\long\def\gcl#1{\mathrm{gcl}\left(#1\right)}
 \global\long\def\Aut{\mathrm{Aut}}
 \global\long\def\Autf{\mathrm{Autf}}
 \global\long\def\acleq{\mbox{\ensuremath{\mathrm{acl}}\ensuremath{\ensuremath{^{eq}}}}}
 \global\long\def\bdd{\mathrm{Bdd}}
 \global\long\def\FK#1#2{\left\langle {#1\atop #2}\right\rangle }
 \global\long\def\PK#1#2{\left({#1\atop #2}\right)}
\global\long\def\Age{\mathrm{Age}}

\title[AUT-Groups of generic structures: (extreme) amenability]{Automorphism Groups of Generic Structures: Extreme Amenability and Amenability}
\authors{
\author{Zaniar Ghadernezhad}
\address{ School of Mathematics\\ Institute for Research in Fundamental Sciences (IPM)\\ P.O. Box 19395-5746\\ Tehran, Iran.}
\email{zaniar.gh@gmail.com}
\author{Hamed Khalilian}
\address{ Department of Mathematical Science\\ Tarbiat Modares University\\ Jalal Ale Ahmad Highway\\ P.O.Box 14115-111\\ Tehran, Iran.}
\email{hamedkhalilian2004@gmail.com}
\author{Massoud Pourmahdian}
\address{ School of Mathematics\\ Institute for Research in Fundamental Sciences (IPM)\\ P.O. Box 19395-5746\\ Tehran, Iran\\ and \\ Department of Mathematics and CS\\ Amirkabir University of Technology\\ P.O.Box 15875-4413 \\ Tehran, Iran.   }
\email{pourmahd@ipm.ir}
}
\maketitle

\begin{abstract}
We investigate correspondences \sloppy between extreme amenability and amenability of automorphism groups
of Fra\"iss\'e-Hrushovski generic structures that are obtained from
smooth classes, and their Ramsey type properties of their smooth classes,
similar to \cite{KPT,TatchMoore2011}. In particular, we focus on
some Fra\"iss\'e-Hrushovski generic structures that are obtained
from pre-dimension functions. Using these correspondences, we prove that automorphism groups
of ordered Hrushovski generic graphs are not extremely amenable in
both cases of collapsed and uncollapsed. Moreover, we prove that automorphism
groups of Fra\"iss\'e-Hrushovski generic structures that are obtained
from pre-dimension functions with rational coefficients are not amenable.
\end{abstract}

\maketitle
\tableofcontents{}

\section{Introduction}

An extensive research has been devoted to studying dynamical properties
of automorphism groups of the Fra\"iss\'e-limit of a class of finite
structures satisfying joint embedding and amalgamation properties.
Suppose $\mathcal{K}$ is a class of finite structures in a relational
language $\mathfrak{L}$ with the joint embedding (JEP), the amalgamation
(AP) and the hereditary (HP) properties. It is well-known that there
exits a unique ultra-homogeneous countable structure $\mathbf{M}$ whose class of finite
substructures up to isomorphism is $\mathcal{K}$
(see \cite{MR1462612,KPT} for more information). The structure $\mathbf{M}$
is called the Fra\"iss\'e-limit of $\mathcal{K}$.

A rich model theoretic studies has been developed for understanding
the first-order theory of $\mathbf{M}$. In particular, the automorphism
groups of Fra\"iss\'e-limit structures have been recently of central
attention. It is well-known that the
automorphism group of $\mathbf{M}$, denoted by $\Aut\left(\mathbf{M}\right)$,
is a Polish closed subgroup of the permutation group of its underlying
set i.e. it can be seen as closed subgroup of $S_{\omega}$. A good survey for various kind of questions and results in the topic can be found in \cite{MacS}.

In the seminal paper \cite{KPT} of Kechris, Pestov and Todorcevic
a close correspondence between extreme amenability of $\Aut\left(\mathbf{M}\right)$
and certain combinatorial property of the class $\mathcal{K}$, called
the \emph{Ramsey property}, has been discovered. Let $G$ be a topological
group. A continuous action $\Gamma$ of $G$ on a compact Hausdorff
space $X$ is called a \emph{$G$-flow}. Group $G$ is called \emph{extremely
amenable} if every $G$-flow $\left(G,\Gamma,X\right)$ has a fix
point in $X$.

In \cite{KPT}, they have shown that the automorphism group of an
ordered Fra\"iss\'e-limit structure $\mathbf{M}$ is extremely amenable
if and only if its ordered Fra\"iss\'e-class has the Ramsey property.
Later in \cite{TatchMoore2011}, a connection has been found between
amenability of $\Aut\left(\mathbf{M}\right)$ and another combinatorial
property called the \emph{convex Ramsey property.} A Hausdorff topological
group $G$ is \emph{amenable} if every $G$-flow $\left(G,\Gamma,X\right)$
supports an $G$-invariant Borel probability measure on $X$. It has
been proved in \cite{TatchMoore2011} that $\Aut\left(\mathbf{M}\right)$
is amenable if and only if $\mathcal{K}$ has the convex Ramsey property.

\sloppy Our paper follows similar paths for adopting their line of research for Fra\"iss\'e-Hrushovski limits of smooth classes. The
Fra\"iss\'e-Hrushovski limits of smooth classes includes the original
construction of Hrushovski of CM-trivial strongly minimal sets \cite{Hrunew},
which are very important structures model-theoretically, as well as
the original Fra\"iss\'e-limits of finite structures. A class of
(finite) structures $\mathcal{K}$ together with a partial ordering
$\leqslant$ is called a \emph{smooth class} if for every $A,A_{1},A_{2}\in\mathcal{K}$
with $A_{1},A_{2}\subseteq A$ whenever $A_{1}\leqslant A$, it follows
that $A_{1}\cap A_{2}\leqslant A_{2}$ (see Definition $\ref{def:smooth}$).
It is worth noting that the notion of substructure satisfies this
condition. Similar to the Fra\"iss\'e-limit case, one can show that
for a smooth class $\left(\mathcal{K},\leqslant\right)$ with HP and the adopted
JEP and AP, there is a $\left(\mathcal{K},\leqslant\right)$-\emph{generic structure} (see Proposition $\ref{prop:sgen}$). A natural question is
to verify (extreme) amenability of the automorphism group of a $\left(\mathcal{K},\leqslant\right)$-generic
structure.

In this paper in Section $\ref{sec:2}$, firstly we show that indeed
a similar correspondence of \cite{KPT} between extreme amenability
of the automorphism group of a $\left(\mathcal{K},\leqslant\right)$-generic
structure, and a modified definition of Ramsey property for $\left(\mathcal{K},\leqslant\right)$
is valid. Later, this correspondence enables us to show that for each $\alpha\geq1$ the automorphism
groups of ordered ab-initio generic graphs $\mathbf{M}_{\alpha}$, ordered collapsed generic graphs $\mathbf{M}_\alpha^\mu$ and ordered $\omega$-categorical generics $\mathbf{M}_\alpha^{\mathsf f}$, are not extremely amenable.

In Section $\ref{sec:3}$, we prove a similar correspondence of \cite{TatchMoore2011}
between amenability of the automorphism group of a $\left(\mathcal{K},\leqslant\right)$-generic
structure, and again a modified version of convex Ramsey property for
$\left(\mathcal{K},\leqslant\right)$. This helps us, in Section $\ref{sec:4}$,
to rule out the amenability of the automorphism group of the ab-initio
generic structures that are obtained from pre-dimension functions
with rational coefficients. However, the amenability question of the
ab-initio generic structures that are obtained from pre-dimension
functions with irrational coefficients and $\omega$-categorical Hrushovski
generic structures remain unanswered in this manuscript\footnote{ After the earlier version of this paper was uploaded in ArXiv, David M. Evans in an email correspondence informed us that he can show, using a different method, the automorphism groups of generic structures that are obtained from pre-dimension
functions with irrational coefficients and the $\omega$-categorical generic structures are not amenable.}.

\noindent\textit{Acknowledgement:\/} The authors would like to thank the anonymous referee for the encouraging comments and the thoughtful suggestions.

\section{Extreme amenability of automorphism groups of generic structures}

\label{sec:2}

In \cite{KPT}, the general correspondence between the extreme amenability
of the automorphism group of an ordered Fra\"iss\'e structure and
the Ramsey property of its finite substructures has been discovered.
In this section, in Theorem $\ref{rams}$, we prove that indeed a
similar correspondence for the automorphism group of a generic structure
and the Ramsey property of its $\leqslant$-closed finite substructures
holds. Below some backgrounds about the smooth classes and Fra\"iss\'e-Hrushovski
limits is presented.

\subsection{Background}

\subsubsection{Smooth class}
\begin{defn}
\label{def:smooth} Let $\mathcal{\mathfrak{L}}$ be a finite relational
language and $\mathcal{K}$ be a class of $\mathfrak{L}$-structures
which is closed under isomorphism and substructure. Let $\leqslant$
be a reflexive and transitive relation on elements of $A\subseteq B$
of $\mathcal{K}$ and moreover, invariant under $\mathfrak{L}$-embeddings
such that it has the following properties:
\begin{enumerate}
\item $\emptyset\in\mathcal{K}$, and  $\emptyset\leqslant A$ for all $A\in\mathcal{\mathcal{K}}$;
\item If $A\subseteq B^{\prime}\subseteq B$, then $A\leqslant B$ implies
that $A\leqslant B^{\prime}$;
\item If $A,A_{1},A_{2}\in\mathcal{K}$ and $A_{1},A_{2}\subseteq A$, then
$A_{1}\leqslant A$ implies $A_{1}\cap A_{2}\leqslant A_{2}$.
\end{enumerate}
The class $\mathcal{K}$ together with the relation $\leqslant$ is
called a \emph{smooth class}. For $A,B\in\mathcal{\mathcal{K}}$ if $A\leqslant B$, then we say
that $A$ is\emph{ $\leqslant$-closed substructure} of $B$, or simply
$A$ is\emph{ $\leqslant$-closed }in $B$. Moreover, if $N$
is an infinite $\mathfrak{L}$-structure such that $A\subseteq N$,
we denote $A\leqslant N$ whenever $A\leqslant B$ for every
finite substructure $B$ of $N$ that contains $A$. We say an embedding
$\Gamma$ of $A$ into $N$ is \emph{$\leqslant$-embedding}
if $\Gamma\left[A\right]\leqslant N$.
\end{defn}
\noun{Notation. }Suppose $A,B,C$ are $\mathfrak{L}$-structures with
$A,B\subseteq C$. We denote $AB$ for the $\mathfrak{L}$-substructure
of $C$ with domain $A\cup B$. For an $\mathfrak{L}$-structure $ N$,
denote $\mbox{Age}\left( N\right)$ for the set of all finite
substructures of $ N$; up to isomorphism.
\begin{defn}
Let $\left(\mathcal{K},\leqslant\right)$ be a smooth class.
\begin{enumerate}
\item We say $\left(\mathcal{K},\leqslant\right)$ has the hereditary property
(HP) if $A\in\mathcal{K}$ and $B\subseteq A$, then $B\in\mathcal{K}$.
\item Suppose $A,B$ and $C$ are elements of $\mathcal{K}$ such that $A\leqslant B,C$.
The \emph{free-amalgam} of $B$ and $C$ over $A$ is a structure
with domain $BC$ whose only relations are those from $B$ and $C$
such that $B\cap C=A$. We denote it by $B\otimes_{A}C$.
\item We say $\left(\mathcal{K},\leqslant\right)$ has the \emph{$\leqslant$-amalgamation
property} (AP) if for every $A,B,C\in \mathcal K$ and $\leqslant$-embeddings $\gamma_1:A\rightarrow B$ and $\gamma_2:A\rightarrow C$, there are  $D$ and $\leqslant$-embeddings $\lambda_1:B\rightarrow D$ and $\lambda_2:C\rightarrow D$ such that $\lambda_1\circ \gamma_1=\lambda_2\circ \gamma_2$ (equivalently; as we assume  $\mathcal K$ is closed under isomorphism,  for every $B,C\in \mathcal K$ that have a common substructure $A$ with $A\leqslant B,C$, there is $D\in\mathcal{K}$ such that $B\leqslant D$
and $C\leqslant D$).
\item We say $\left(\mathcal{K},\leqslant\right)$ has the \emph{free-amalgamation
property} if for $B,C\in \mathcal K$ that have a common substructure $A$
with $A\leqslant B,C$, then $B\otimes_{A}C\in\mathcal{K}$.
\end{enumerate}
\end{defn}
\begin{rem}
We included $\emptyset$ in the class $\mathcal{K}$ in order to consider
the joint embedding property (JEP) as a special case of the $\leqslant$-amalgamation
property.\end{rem}
\begin{prop}
\label{prop:sgen} If $\left(\mathcal{K},\leqslant\right)$ is a smooth
class with the $\leqslant$-amalgamation property, then there is a
unique countable structure $\mathbf{M}$, up to isomorphism, satisfying:
\begin{enumerate}
\item $\Age\left(\mathbf{M}\right)=\mathcal{K}$;
\item $\mathbf{M}=\bigcup_{i\in \omega}A_i$ where $A_i\in \mathcal K$ and $A_i\leqslant A_{i+1}$ for every $i\in \omega$;
\item If $A\leqslant\mathbf{M}$ and $A\leqslant B\in\mathcal{K}$, then
there is an embedding $\Gamma :B\longrightarrow\mathbf{M}$ with $\Gamma\upharpoonright_{A}=\mbox{id}_{A}$
and $\Gamma\left[B\right]\leqslant\mathbf{M}$.
\end{enumerate}
\end{prop}
\begin{proof}
See \cite{KueLas}. \end{proof}
\begin{defn}
The structure $\mathbf{M}$, that is obtained in the above proposition,
is called the \emph{Fra\"iss\'e-Hrushovski }$\left(\mathcal{K},\leqslant\right)$\emph{-generic
structure} or simply $\left(\mathcal{K},\leqslant\right)$\emph{-generic
structure.}\end{defn}
\begin{fact}
\label{f5} (See \cite{KueLas}) Suppose $A\subseteq_{fin}\mathbf{M}$.
Then, there is a unique smallest finite closed set that contains $A$
in $\mathbf{M}$. It is called $\leqslant$-closure of $A$ in $\mathbf{M}$
that is denoted by $\cl A$.
\end{fact}

\subsubsection{Ab-initio classes of graphs}

\label{sub:ab}

Let $\mathcal{\mathfrak{L}}=\left\{ \mathfrak{R}\right\} $ consist
of a binary relation $\mathfrak{R}$ and let $\mathcal{K}$ be the
class of all finite graphs. For $\alpha\geq1$, define $\delta_{\alpha}:\mathcal{K}\longrightarrow\mathbb{R}$
as $\delta_{\alpha}\left(A\right)=\alpha\cdot\left|V\left(A\right)\right|-\left|E\left(A\right)\right|$
where $V\left(A\right)$ is the set of vertices of $A$ and $E\left(A\right)$
the set of edges of $A$. For every $A\subseteq B\in\mathcal{K}$,
define $A\leqslant_{\alpha}B$ if
\[
\delta_{\alpha}\left(C/A\right):=\delta_{\alpha}\left(C\right)-\delta_{\alpha}\left(A\right)\geq0,
\]
for every $C$ with $A\subseteq C\subseteq B$. Finally put $\mathcal{K}_{\alpha}^{+}:=\left\{ A\in\mathcal{K}:\delta_{\alpha}(B)\geq0,\mbox{ for every }B\subseteq A\right\} $.
\begin{fact}
$\left(\mathcal{K}_{\alpha}^{+},\leqslant_{\alpha}\right)$ is a smooth
class with the free-amalgamation property and HP.
\end{fact}
Hence, there is the unique countable $\left(\mathcal{K}_{\alpha}^{+},\leqslant_{\alpha}\right)$-generic structure $\mathbf{M}_\alpha$.
When the coefficient $\alpha$ is rational, using a finite-to-one function $\mu$ over the $0$-minimally algebraic elements (see Definition \ref{zero-min}), one can restrict the ab-initio class $\mathcal{K}_{\alpha}^{+}$ to $\mathcal{K}_{\alpha}^{\mu}$ such that $\left(\mathcal{K}_{\alpha}^{\mu},\leqslant_{\alpha}\right)$ has AP (see \cite{Balshi}).
\begin{defn}
\label{zero-min} Suppose $\alpha \geq 1$ is a rational number and $\mathbf{M}_\alpha$ is the $\left(\mathcal{K}_{\alpha}^{+},\leqslant_{\alpha}\right)$-generic structure:
\begin{enumerate}
 \item Suppose $A$ and $B$ are two disjoint finite sets in $\mathbf{M}_{\alpha}$.
B is called \emph{$0$-algebraic over} A if $\delta\left(B/A\right)=0$
and $\delta\left(B_{0}/A\right)>0$ for all proper subset $\emptyset\neq B_{0}\subsetneqq B$.
$B$ is called \emph{$0$-minimally algebraic over} $A$ if there
is no proper subset $A_{0}$ of $A$ such that $B$ is $0$-algebraic
over $A_{0}$.
 \item Let $\mathcal E\in \mathcal{K}_{\alpha}^{+}\times \mathcal{K}_{\alpha}^{+}$ be the set of all $(A,B)$ such that $B$ is $0$-minimally algebraic over $A$ and $A\neq \emptyset$. Define a function $\mu:\mathcal E\rightarrow \mathbb{N}$	such that $\mu $ is finite-to-one, and $\mu(A,B)\geq \delta(A)$ for every $(A,B)\in \mathcal E$.
 \item Let $\mathcal{K}_{\alpha}^{\mu}\subseteq \mathcal{K}_{\alpha}^{+}$ be such that $A\in \mathcal{K}_{\alpha}^{\mu}$ if for every $A'\subseteq A$ and $B'$, a 0-minimally algebraic set over $A'$, the number of pairwise disjoint isomorphic copies of $B'$ over $A'$ in $A$ is bounded by $\mu(A',B')$.
\end{enumerate}
 \end{defn}

\begin{fact}
(cf.	\cite{Hrunew}) The class $\left(\mathcal{K}_{\alpha}^{\mu},\leqslant_{\alpha}\right)$ is a smooth
class with AP and HP.
\end{fact}
Let $\mathbf{M}_\alpha^\mu$ for  the $\left(\mathcal{K}_{\alpha}^{\mu},\leqslant_{\alpha}\right)$-generic structure.\\
To obtain an $\omega$-categorical generic structure one needs further restrictions: Suppose $\mathsf f:\mathbb R^{\geq 0}\rightarrow \mathbb R^{\geq 0}$ is an increasing unbounded function. Then let
$$\mathcal K_\alpha^{\mathsf f}:=\left\{A\in \mathcal K_\alpha^{+}:\delta_\alpha\left(A'\right)\geq \mathsf f\left(\left|A'\right|\right) \forall A'\subseteq A\right\}.$$
\begin{fact}
For suitable choice of $\mathsf f$ (called \emph{good}) the class $\left(\mathcal K_\alpha^{\mathsf f},<_\alpha\right)$ a smooth
class with the free-amalgamation property and HP, where $A<_\alpha B$ iff $\delta_\alpha\left(A\right)<\delta_\alpha\left(B'\right)$ for every $A\subsetneqq B'\subseteq B$.
\end{fact}
Hence, for a good $\mathsf f$, there is the countable $\left(\mathcal{K}_{\alpha}^{\mathsf f},<_{\alpha}\right)$-generic structure $\mathbf{M}_\alpha^{\mathsf f}$. Moreover, the generic structure $\mathbf{M}_\alpha^{\mathsf f}$ is an $\omega$-categorical structure (see \cite{E:pre} for more details).

\subsection{$\leqslant$-Ramsey property and its correspondence with extreme
amenability }

Denote $S_{\omega}$ for the set of all permutations of $\mathbb{N}$.
It is a well-known fact that $S_{\omega}$ with the point-wise convergence
topology forms a Polish group. From now on, we consider the point-wise
convergence topology on $S_{\omega}$.
\label{sub:rams}
\begin{defn}
For a topological group $G$ and a subgroup $H$ of $G$ by a \emph{$k$-coloring}
$c$ of $G/H$ with $k\in\mathbb{N}\backslash\left\{ 0\right\} $,
we mean a map $c:\left\{ hH:h\in G\right\} \longrightarrow\left\{ 0,1,2,\cdots,k-1\right\} $,
from the set of left cosets of $H$ into $\left\{ 0,1,\cdots,k-1\right\} $.
\end{defn}

\begin{fact}
\label{2} (See \cite{KPT} Proposition 4.2) Let $G$ be a closed
subgroup of $S_{\omega}$. Then, the followings are equivalent:
\begin{enumerate}
\item $G$ is extremely amenable;
\item For any open subgroup $H$ of $G$, any $k$-coloring $c:G/H\longrightarrow\left\lbrace 0,1,\cdots,k-1\right\rbrace $
and any finite $A\subseteq_{\text{fin}}G/H$, there are $g\in G$
and $i\in\left\{ 0,1,\cdots,k-1\right\} $ such that $c\left(g.a\right)=i$,
for all $a\in A$.
\end{enumerate}
\end{fact}
We work with a fixed smooth class $\left(\mathcal{K},\leqslant\right)$
with AP and HP. Let $\mathbf{M}$ be the countable $\left(\mathcal{K},\leqslant\right)$-generic
structure with $\mathbb N$ as the underlying universe. Put $G:=\Aut\left(\mathbf{M}\right)$.
It is also well-known that $G$ is a closed subgroup of $S_{\omega}$.
Let $A\subseteq\mathbf{M}$ be a finite subset of $\mathbf{M}$.
Write
\[
G_{\left(A\right)}:=\left\{ g\in G:\hspace*{0.5cm}g\left(a\right)=a,\mbox{ for all }a\in A\right\} ,
\]
for the \emph{point-wise stabilizer} of $A$ in $G$, and write
$G_{A}=\left\lbrace g\in G:g\left[A\right]=A\right\rbrace $ for the
\emph{set-wise stabilizer} of $A$ in $G$ where $\emptyset\neq A\subseteq_{fin}\mathbf{M}$.
\begin{rem}
Note that $\left\{ G_{(A)}:\emptyset\neq A\leqslant\mathbf{M}\right\} $
forms a basis of neighborhood of $1_{G}$. \end{rem}
\begin{defn}
Suppose $A\in\mathcal{K}$ and let $ N$ is any $\mathfrak{L}$-structure.
We denote $\PK{ N}A$ for the set of all $\leqslant$-embeddings
of $A$ into $ N$. For $k\in\mathbb{N}\backslash\left\{ 0\right\} $,
we call a function $c:\PK NA\rightarrow\left\{ 0,1,\cdots,k-1\right\} $\emph{
}a\emph{ $k$-coloring function.}\end{defn}
Suppose $A\in\mathcal{K}$. The group $G$ acts naturally on $\PK{\mathbf{M}}A$
in the following way: When $\Gamma\in\PK{\mathbf{M}}A$
\[
g\cdot\Gamma:=\Gamma'
\]
if $\Gamma'\left(A\right)=g\left[\Gamma\left(A\right)\right]$. It
is worth noting, since elements of $G$ sends $\leqslant$-closed
sets to $\leqslant$-closed sets, this action is well-defined.
\begin{defn}
We say that $G$ \emph{preserves a linear ordering} $\preceq$ on
$\mathbf{M}$ if $a\preceq b$, implies $g\left(a\right)\preceq g\left(b\right)$,
for every $a,b\in\mathbf{M}$ and $g\in G$. \end{defn}
\begin{prop}
\label{3} The following conditions are equivalent:
\begin{enumerate}
\item $G$ is extremely amenable;
\item
\begin{enumerate}
\item $G_{\left(A\right)}=G_{A}$, for any finite $\emptyset\neq A\leqslant\mathbf{M};$
\item \sloppy For every $A\leqslant B\in\mathcal{K}$  and every $k$-coloring function $c:\PK{\mathbf{M}}A\longrightarrow\left\lbrace 0,1,\cdots,k-1\right\rbrace $,
there are $\Lambda\in\PK{\mathbf{M}}B$ and $0\leq i\leq k-1$ such
that $c\left(\Lambda\circ\Gamma\right)=i$, for all $\Gamma\in\PK{\Lambda\left(B\right)}A$.
\end{enumerate}
\item
\begin{enumerate}
\item $G$ preserves a linear ordering;
\item As 2-(b) above.
\end{enumerate}
\end{enumerate}
\end{prop}
\begin{proof}
(Similar to the proof of Proposition 4.3. in \cite{KPT}) An easy
argument shows that (2) and (3) are equivalent.\\
 $1\Longrightarrow3$. Assume that $G$ is extremely amenable. Since
$LO$, the space of invariant linear orderings defined on $\mathbf{M}$,
forms a $G$-flow, it follows that the action of $G$ on $LO$ has
a fixed point. This is exactly our expected ordering.\\
To show 3-(b), fix $A\leqslant B\in\mathcal{K}$ and suppose $c:\PK{\mathbf{M}}A\longrightarrow\left\lbrace 0,1,\cdots,k-1\right\rbrace $
is a $k$-coloring. Fix $\Lambda$ a $\leqslant$-embedding of $B$
in $\mathbf{M}$ and let $B_{0}:=\Lambda\left(B\right)$. Then, $B_{0}\leqslant\mathbf{M}$.
Take $A_{0}\leqslant B_{0}$ to be the corresponding $\leqslant$-closed
copy of $A$ inside $B_{0}$ under $\Lambda$. Put $H=G_{A_{0}}=G_{\left(A_{0}\right)}$
which is an open subgroup of $G$. We can identify the set $G/H$
of left cosets of $H$ in $G$ with $\PK{\mathbf{M}}{A_{0}}$. Now
by applying Fact \ref{2} to $H$, $c$ and $\PK{\mathbf{M}}{A_{0}}$,
one can find $i$ with $0\leq i\leq k-1$ and $g\in G$ such that
$c\left(g\cdot\left(\Lambda\circ\gamma\right)\right)=i$, for all
$\gamma\in\PK{B_{0}}{A_{0}}$. Let $B'=g\left[B_{0}\right]=g\left[\Lambda\left(B\right)\right]\leqslant\mathbf{M}$.
Pick $\Lambda'\in\PK{\mathbf{M}}B$ such that $\Lambda'\left(B\right)=g\left[\Lambda\left(B\right)\right]$.
Then, $c\left(\Lambda'\circ\gamma\right)=i$, for any $\gamma\in\PK{B'}A$.\\
 $2\Longrightarrow1$. A similar argument as above shows ($2$) implies
Fact \ref{2}.2. Hence, $G$ is extremely amenable.\\
\end{proof}
\begin{defn}
\label{ram} Assume $A\leqslant B\leqslant C\in\mathcal{K}$ and $k\geq1$.
We write
\[
C\longrightarrow\left(B\right){}_{k}^{A},
\]
if for every $k$-coloring $c:\PK CA\longrightarrow\left\lbrace 0,1,\cdots,k-1\right\rbrace $,
there exists $\lambda\in\PK CB$ such that $c\left(\lambda\circ\gamma\right)$
is constant for all $\gamma\in\PK{\lambda\left(B\right)}A$. In this
case $\lambda\left(B\right)$ is called a \emph{$c$-monochromatic}
copy of $B$ in $C$. We say that the class $\left(\mathcal{K},\leqslant\right)$
has the \emph{$\leqslant$-Ramsey property} if for every $A\leqslant B\in\mathcal{K}$
and $k\geq2$, there exists $C\in\mathcal{K}$ with $B\leqslant C$
such that $C\longrightarrow\left(B\right){}_{k}^{A}.$\end{defn}
\begin{rem}
\label{2=00003Dk} A similar argument as in classical Ramsey theory
shows that if for every $A\leqslant B\in\mathcal{K}$ we have $C\longrightarrow\left(B\right){}_{2}^{A}$
for some $C\in\mathcal{K}$, then the smooth class $\left(\mathcal{K},\leqslant\right)$
has the $\leqslant$-Ramsey property (See \cite{Prom}, page 81-82).
\end{rem}
The following theorem gives the main correspondence that we have mentioned
in the introduction. The proof is similar to the proof of Proposition
4.5. in \cite{KPT}, and we give the analogues modification of the
proof in order to highlight the role of $\leqslant$ relation in the
$\left(\mathcal{K},\leqslant\right)$-generic structure.
\begin{thm}
\label{rams} The followings are equivalent:
\begin{enumerate}
\item $G$ is extremely amenable;
\item
\begin{enumerate}
\item $G$ preserves a linear ordering;
\item $\left(\mathcal{K},\leqslant\right)$ has the $\leqslant$-Ramsey
property.
\end{enumerate}
\end{enumerate}
\end{thm}
\begin{proof}
$1\Longrightarrow2$. We have already presented a proof for 2-(a).
Now we are going to show the $\leqslant$-Ramsey property, assuming
the extremely amenability of $G$. By Remark \ref{2=00003Dk}, we
only need to check the $\leqslant$-Ramsey property for $k=2$. Suppose,
on the contrary that, there are $A\leqslant B\in\mathcal{K}$ such
that $C\nrightarrow\left(B\right){}_{2}^{A}$ for all $C\in\mathcal{K}$. Pick $\Lambda\in\PK{\mathbf{M}}B$ and let $B_{0}:=\Lambda\left(B\right)$. Then, for every finite $E\leqslant\mathbf{M}$ with $B_{0}\leqslant E$, there exists a $2$-coloring $c_{E}:\PK EA\longrightarrow\left\lbrace 0,1\right\rbrace $
such that for any $\lambda^{\prime}\in\PK EB$ the value of $c_{E}\left(\lambda'\circ\gamma\right)$
is not constant when $\gamma\in\PK{\lambda'\left(B\right)}A$ varies. \\
Take $\mathcal{I}:=\left\{ F\leqslant M:F\subseteq_{fin}\mathbf{M}\right\} $
as an index set and for $D\in\mathcal{I}$, \sloppy let $\mathcal{X}_{D}:=\left\lbrace F\in\mathcal{I}:D\leqslant F\right\rbrace$.
From Fact $\ref{f5}$, it follows that $\mathcal{E}:=\left\lbrace \mathcal{X}_{A}:A\in\mathcal{I}\right\rbrace $
has the finite intersection property. Hence, there exists an ultra-filter
$\mathcal{U}$ on the index set $\mathcal{I}$ such that for every
finite $D\in\mathcal{I}$ the set $\mathcal{X}_{D}\in\mathcal{U}$. For each $\Gamma\in\PK{\mathbf{M}}A$ exactly one of the followings
cases holds:
\begin{enumerate}
\item $\left\lbrace E\in\mathcal{I}:\cl{\Gamma\left(A\right)B_{0}}\leqslant E\mbox{ and }c_{E}\left(\Gamma\right)=0\right\rbrace \in\mathcal{U}$,
or
\item $\left\lbrace E\in\mathcal{I}:\cl{\Gamma\left(A\right)B_{0}}\leqslant E\mbox{ and }c_{E}\left(\Gamma\right)=1\right\rbrace \in\mathcal{U}$.
\end{enumerate}
Define a $2$-coloring $c:\PK{\mathbf{M}}A\longrightarrow\left\{ 0,1\right\} $
as follows: for $\Gamma\in\PK{\mathbf{M}}A$
\[
c\left(\Gamma\right):=i\Leftrightarrow\left\lbrace E\in\mathcal{I}:\cl{\Gamma\left(A\right)B_{0}}\leqslant E\mbox{ and }c_{E}\left(\Gamma\right)=i\right\rbrace \in\mathcal{U}.
\]
Now by Proposition \ref{3}.2-(b), there are $\Lambda'\in\PK{\mathbf{M}}B$
and $i\in\left\{ 0,1\right\} $ such that $c\left(\Lambda'\circ\gamma\right)=i$,
for all $\gamma\in\PK{\Lambda'\left(B\right)}A$. For each $\gamma\in\PK{\Lambda'\left(B\right)}A$
put
\[
\mathcal{A}_{\gamma}=\left\lbrace E\in\mathcal{I}:\cl{\Lambda'\left(B\right)B_{0}}\leqslant E\mbox{ and }c\left(\Lambda'\circ\gamma\right)=c_{E}\left(\Lambda'\circ\gamma\right)=i\right\rbrace .
\]
Note that both sets
\[
\mathcal{I}_{B_{0},\gamma}:=\left\lbrace E\in\mathcal{I}:\cl{\Lambda'\circ\gamma\left(A\right)B_{0}}\leqslant E\mbox{ and }c_{E}\left(\Lambda'\circ\gamma\right)=i\right\rbrace
\]
and
\[
\mathcal{X}_{\cl{\Lambda'\left(B\right)B_{0}}}=\left\lbrace E\in\mathcal{I}:\cl{\Lambda'\left(B\right)B_{0}}\leqslant E\right\rbrace
\]
are in $\mathcal{U}$. Furthermore, $\mathcal{A}_{\gamma}\supseteq\mathcal{I}_{B_{0},\gamma}\cap\mathcal{X}_{\cl{\Lambda'\left(B\right)B_{0}}}$.
Hence, $\mathcal{A}_{\gamma}\in\mathcal{U}$. Let $E\in\bigcap_{\gamma\in\PK{\Lambda'\left(B\right)}A}\mathcal{A}_{\gamma}\ne\emptyset$.
Note that if $B_{0}\leqslant E$ then for each $\gamma\in\PK{\Lambda'\left(B\right)}A$,
 $c_{E}\left(\gamma\right)=i$. Therefore, $\lambda\left(B\right)$
is a monochromatic subset with respect to $c_{E}$ where $\lambda\in\PK EB$
which is a contradiction.\\
 $2\Longrightarrow1$ Part (b) of 2 in Proposition \ref{3} trivially
follows from the $\leqslant$-Ramsey property and part (a) of 2 in
Proposition \ref{3} follows from 2-(a). Hence, $G$ is extremely
amenable.
\end{proof}

\subsection{$\leqslant$-Ramsey property for some ab-initio classes}

In this subsection, we show certain ab-initio classes obtained from pre-dimension functions does not have the $\leqslant$-Ramsey property.

Recall the followings from \cite{Prom}.
\begin{defn}
Suppose $A=\left(V,E\right)$ is a graph where $V$ is the set of
vertices and $E$ is set of edges of $A$.
\begin{enumerate}
\item Let $e_{A}:=\mid E\left(A\right)\mid$ and $v_{A}:=\mid V\left(A\right)\mid$.
\item For a vertex $a\in A$, $\mbox{deg}\left(a\right)$ denotes the degree
of $a$ in $A$.
\item The \emph{maximum density} of $A$, denoted by $m\left(A\right)$,
\sloppy is defined as $m\left(A\right):=\mbox{max}\left\{ \dfrac{e_{B}}{v_{B}};B\subseteq A\right\}$.
\item Let $\eta\left(A\right):=\mbox{min}\left\{ \mbox{deg}\left(a\right):a\in V\left(A\right)\right\} $
and define $\eta^{*}\left(A\right):=\mbox{max}\left\{ \eta\left(B\right):B\subseteq A\right\} $.
\item For two graphs $B\subseteq C$, we abbreviate $C\stackrel{v}{\rightarrow}\left(B\right){}_{r}$
to indicate any vertex $r$-coloring of $C$ has a subgraph $B'$
isomorphic to $B$, whose all vertices are monochromatic.
\end{enumerate}
\end{defn}

\begin{defn}
The smooth class $\left(\mathcal{K},\leqslant\right)$ with HP has
the\emph{ one-point $\leqslant$-Ramsey property} if for every one-point
structure $A$ and every structure $B$ with $A\leqslant B$ and $k\geq2$,
there exists $C\in\mathcal{K}$ with $B\leqslant C$ such that $C\rightarrow\left(B\right){}_{k}^{A}$.
\end{defn}
The following lemma provides the key idea for proving Theorem \ref{nram}.
For its proof, the reader is referred to \cite{Prom} (Lemma 12.2,
page 130).
\begin{lem}
\label{4} Suppose $B$ and $C$ are two graphs such that $m\left(C\right)<\frac{1}{2}\cdot r\cdot\eta^{*}\left(B\right)$,
for some $r\geq2$. Then, $C\stackrel{v}{\nrightarrow}\left(B\right){}_{r}$. \end{lem}
\begin{thm}
\label{nram} Suppose $\left(\mathcal K,\sqsubseteq\right)\in \left\{\left(\mathcal K_\alpha^+,\leqslant_\alpha\right),\left(\mathcal K_\alpha^\mu,\leqslant_\alpha\right),\left(\mathcal K_\alpha^f,<_\alpha\right)\right\}$ for $\alpha\geq1$. Then the class $\left(\mathcal K,\sqsubseteq\right)$
does not have the one-point $\sqsubseteq$-Ramsey property. \end{thm}
\begin{proof}
Let $A$ be the singleton graph and let $L$ be a loop with $n$ vertices,
$n\geq3$ that contains $A$. It is easy to see that $L\in\mathcal{K}$,
for every $\alpha\geq1$. Furthermore, any embedding of $A$ in $L$
is $\sqsubseteq$-closed. In particular, $A\sqsubseteq L$.
An easy calculation shows for any $C\in\mathcal{K}$,
we have $m\left(C\right)\leq\alpha$. Moreover, $\eta^{*}\left(L\right)=2$. \\
Now by choosing $r>\alpha$ we have $m\left(C\right)<\frac{1}{2}\cdot r\cdot\eta^{*}\left(L\right)$,
for every $C\in\mathcal{K}$. Hence, in the light of
Lemma \ref{nram}, $C\stackrel{v}{\nrightarrow}\left(L\right){}_{r}$
for every $C\in\mathcal{K}$. On the other hand, since
all embeddings of $A$ inside $L$ are $\sqsubseteq$-closed,
it follows that $C\nrightarrow\left(L\right){}_{r}^{A}$, for every
$C\in\mathcal{K}$. Hence, the class $\left(\mathcal{K},\sqsubseteq \right)$
does not have the one-point $\sqsubseteq$-Ramsey property.
\end{proof}
In fact, the above proof shows something stronger. Suppose $\mathfrak{\mathfrak{L}^*}$
is a finite expansion of $\mathfrak{L}$ that contains a binary relation
$\prec$. Let $\mathcal{K}_{\mathfrak{L}^*}$ be the all
$\mathfrak{L}^*$-expansions of structures $C\in\mathcal{K}$,
in which the relation $\prec$ is interpreted as a linear-ordering
on the universe of $C$. Subsequently, for $A^{\mathfrak{L}},B^{\mathfrak{L}}\in\mathcal{K}_{\mathfrak{L}^*}$
we define $A^{\mathfrak{L}}\sqsubseteq^{*}B^{\mathfrak{L}}$
if and only if $A^{\mathfrak{L}^*}\subseteq B^{\mathfrak{L}^*}$ and $A\sqsubseteq B$,
where $A^{\mathfrak{L}^*}$ and $B^{\mathfrak{L}^*}$ are $\mathfrak{L}^*$-expansions
of graphs $A$ and $B$; respectively.
\begin{cor}
The class $\left(\mathcal{K}_{\mathfrak{L}^*},\sqsubseteq^{*}\right)$
does not have the one-point $\sqsubseteq^{*}$-Ramsey property.
\end{cor}
In this case, the class $\left(\mathcal{K}_{\mathfrak{L}^*},\sqsubseteq^{*}\right)$
has the JEP, HP and AP. Therefore, if we take the $\left(\mathcal{K}_{\mathfrak{L}^*},\sqsubseteq^{*}\right)$-generic
structure $\mathbf{M}_{\mathfrak{L}^*}$, then it is easy
to see that this structure is formed by adding a generic linear ordering
to the $\left(\mathcal{K},\sqsubseteq\right)$-generic
structure $\mathbf{M}$. Note that by our notation $\mathbf{M}\in \left\{\mathbf{M}_\alpha,\mathbf{M}_\alpha^\mu,\mathbf{M}_\alpha^{\mathsf f}\right\}$. In the light of Theorem
\ref{rams} and the corollary above, the following theorem where $\mathfrak R$ is binary is established. It has to be noted that it seems in the case of hypergraphs (relations with arity $>$2) a similar result to Lemma \ref{4} is true.
\begin{thm}
\label{thm:extamngraph} Let $G$ be $\Aut {\left(\mathbf{M}_{\mathfrak L^*}\right)}$ where $\mathfrak L^*$ is a finite expansion $\mathfrak L$ that contains a linear-ordering relation as explained above, and $\mathbf{M}\in \left\{\mathbf{M}_\alpha,\mathbf{M}_\alpha^\mu,\mathbf{M}_\alpha^{\mathsf f}\right\}$ for $\alpha\geq 1$. Then $G$ is not extremely amenable.
\end{thm}

\section{Amenability of automorphism groups of generic structures}

\label{sec:3}

Here, we further continue our project, this time similar to \cite{TatchMoore2011}.
We give the correspondence between amenability of the automorphism
groups of Fra\"iss\'e-Hrushovski generic structures, and the convex
$\leqslant$-Ramsey property of the automorphism group of the structure
with respect to its smooth class, that has been defined later.

In the following subsection, we first adapt the notion of convex Ramsey
property and then study the convex $\leqslant$-Ramsey property with
a slightly different approach. We present the expected correspondence
in Theorem $\ref{thm:cors}$. Later, we prove our main result Theorem
$\ref{thm:main1}$ that shows automorphism groups of generic structures
of certain class of smooth classes does not have the convex $\leqslant$-Ramsey
property and hence they are not amenable. Theorem $\ref{thm:main1}$ provides the ingredient for the next section to investigate the convex $\leqslant_{\alpha}$-Ramsey
property of the $\left(\mathcal{K}_{\alpha},\leqslant_{\alpha}\right)$-generic
structure when the coefficient $\alpha$ is rational.

Throughout this section, $\left(\mathcal{K},\leqslant\right)$ is
a smooth class of finite relational $\mathfrak{L}$-structures with
the $\leqslant$-amalgamation property and HP, and $\mathbf{M}$
is the $\left(\mathcal{K},\leqslant\right)$-generic structure. Suppose
$A\in\mathcal{K}$ and $ N$ is a substructure of $\mathbf{M}$.
Denote $\FK{ N}A$ for the set of all finitely supported
probability measures on $\PK{ N}A$. Suppose $X\leqslant Y\leqslant Z$
are $\leqslant$-closed substructures of $\mathbf{M}$ and let $\mathtt r\in\FK ZY$.
We define $\FK{\mathtt r}X$ to be the set
\[
\left\{ \mathtt q\in\FK ZX:\exists \mathtt p\in\FK YX \forall\Gamma\in\PK YX \forall\Lambda\in\PK ZY,\mathtt q\left(\Lambda\circ\Gamma\right)=\mathtt r\left(\Lambda\right)\cdot \mathtt p\left(\Gamma\right)\right\} .
\]

\subsection{The convex \emph{$\leqslant$-}Ramsey property}
\begin{defn}
We say $\Aut\left(\mathbf{M}\right)$ has the \emph{convex $\leqslant$-Ramsey
property with respect to} $\left(\mathcal{K},\leqslant\right)$ if
for every $A,B\in\mathcal{K}$ with $A\leqslant B$ and every $2$-coloring
function $f:\PK{\mathbf{M}}A\rightarrow\left\{ 0,1\right\} $, there
exists $\mathtt p\in\FK{\mathbf{M}}B$ such that for every $\mathtt q_{1},\mathtt q_{2}\in\FK{\mathtt p}A$,
$\left|f\left(\mathtt q_{1}\right)-f\left(\mathtt q_{2}\right)\right|\leq\frac{1}{2}$.
\end{defn}

If the condition above holds for a coloring function $f$, we say
$f$ satisfies the\emph{ convex $\leqslant$-Ramsey condition}. Note
that the convex \emph{$\leqslant$-}Ramsey property demands that all
coloring functions satisfy the convex \emph{$\leqslant$-}Ramsey condition.\\
Let $A,B\in\mathcal{K}$ such that $A\leqslant B$. Fix an enumeration
$\bar{\eta}=\left(\eta_{1},\cdots,\eta_{m}\right)$ of $\PK BA$ where $\PK BA=\left\{ \eta_{1},\cdots,\eta_{m}\right\} $;
hence $m=\left|\PK BA\right|$. For each $\Lambda\in\PK{\mathbf{M}}B$
define $\Lambda\cdot\bar{\eta}:=\left(\Lambda\circ \eta_{1},\cdots,\Lambda\circ \eta_{m}\right)$,
and note that $\Lambda\circ \eta_{i}\in\PK{\mathbf{M}}A$, for each
$i\in\left\{ 1,\cdots,m\right\} $.
Let $f:\PK{\mathbf{M}}A\rightarrow\left\{ 0,1\right\} $ be a $2$-coloring
function. For each $\Lambda\in\PK{\mathbf{M}}B$ define $f\left(\Lambda\cdot\bar{\eta}\right):=\left(f\left(\Lambda\circ \eta_{1}\right),\cdots,f\left(\Lambda\circ \eta_{m}\right)\right)$
which is a finite sequence of $0$ and $1$ and hence $f\left(\Lambda\cdot\bar{\eta}\right)\in\left\{ 0,1\right\} ^{m}$.\\
Fix $\Lambda\in\PK{\mathbf{M}}B$ and let $\bar{k}:=f\left(\Lambda\cdot\bar{\eta}\right)\in\left\{ 0,1\right\} ^{m}$.
There are two possibilities for elements of $\bar{k}$:

\noun{Case 1: }$k_{i}=k_{j}$ for all $i,j\in\left\{ 1,\cdots,m\right\} $. \\
Then, we can assign a finitely supported probability measure $\mathtt p$
on $\PK{\mathbf{M}}B$ which concentrates on $\Lambda$. 
In which the following corollary holds:
\begin{cor}
\label{cor:monoc} Suppose $A,B\in\mathcal{K}$ and $A\leqslant B$.
Suppose $f:\PK{\mathbf{M}}A\rightarrow\left\{ 0,1\right\} $ is a
coloring function such that there exists $\Lambda\in\PK{\mathbf{M}}B$
where $f\left(\Lambda\cdot\bar{\eta}\right)$ is constant (or monochromatic).
Then, the convex \emph{$\leqslant$-}Ramsey condition holds for $f$.
\end{cor}
\noun{Case 2. }There are $i,j\in\left\{ 1,\cdots,m\right\} $ such
that $k_{i}=0$ and $k_{j}=1$. \\
Then, we can show the following lemma which is needed for Theorem
$\ref{lem:matrixRa}$.
\begin{lem}
\label{lem:free-3} Suppose $\left(\mathcal{K},\leqslant\right)$
has the free-amalgamation property. Then, for every $v\in\left\{ 1,\cdots,m\right\} $
and $w\in\left\{ 0,1\right\} $, there is $\bar{k}^{v,w}:=\left(k_{1}^{v,w},\cdots,k_{m}^{v,w}\right)\in\left\{ 0,1\right\} ^{m}$
with $k_{v}^{v,w}=w$ such that there exist infinitely many distinct
$\Lambda_{v,w}^{s}\in\PK{\mathbf{M}}B$, $s<\omega$ with $f\left(\Lambda_{v,w}^{s}\cdot\bar{\eta}\right)=\bar{k}^{v,w}$. \end{lem}
\begin{proof}
Let $v\in\left\{ 1,\cdots,m\right\} $ and $w\in\left\{ 0,1\right\} $.
Note that $\eta_{v}$ corresponds to a $\leqslant$-embedding of $A$
into $B$. By our assumption $k_{j}=w$ or $k_{i}=w$. Without loss
of generality, let $k_{j}=w.$ By the $\leqslant$-genericity of $\mathbf{M}$
and the free-amalgamation property, there are infinitely many distinct
embeddings $\Lambda_{s}$ for $s<\omega$ such that $\Lambda_{s}\circ \eta_{v}=\Lambda\circ \eta_{v}$.
Since $\Lambda\circ \eta_{v}\left(A\right)$ and $\Lambda\circ \eta_{j}\left(A\right)$
$\leqslant$-closed and isomorphic, there exists an automorphism $g$
of $\mathbf{M}$ such that $g\left[\Lambda\circ \eta_{v}\left(A\right)\right]=\Lambda\circ \eta_{j}\left(A\right)$.
It is easy to check $g\cdot\left(\Lambda_{s}\circ \eta_{v}\right)=\Lambda\circ \eta_{j}$
and hence $f\left(g\cdot\left(\Lambda_{s}\circ \eta_{v}\right)\right)=w$.
Let $\Sigma_{v}:=\left\{ g\cdot\Lambda_{s}:s<\omega\right\}$. Now
consider $\left\{ f\left(\Lambda\cdot\bar{\eta}\right)\in\left\{ 0,1\right\} ^{m}:\Lambda\in\Sigma_{v}\right\} $.
Since $\left|\Sigma_{v}\right|$ is infinite, there is $\bar{k}^{v,w}\in\left\{ 0,1\right\} ^{m}$
such that $f\left(\Lambda_{v,w}^{s}\cdot\bar{\eta}\right)=\bar{k}^{v,w}$
for infinitely many distinct $\Lambda_{v,w}^{s}$ 's in $\Sigma_{v}$.
Note that $k_{v}^{v,w}=k_{j}=w$, and hence we are done.\end{proof}
\begin{rem}
\label{rem:28} In general for any smooth class $\left(\mathcal{K},\leqslant\right)$,
one can easily modify the argument above to guarantee that under the
assumption of AP, there exists at least one embedding with the desired
property.
\end{rem}
Suppose $\mathtt q_{1},\mathtt q_{2}$ are two elements of $\FK{\mathbf{M}}A$.
Let $\left\{ \Gamma_{i}:i\in I\right\} $ to be an enumeration of
all elements $\PK{\mathbf{M}}A$; without repetition. Note that $\left|I\right|=\omega$
when $\left(\mathcal{K},\leqslant\right)$ has the free-amalgamation
property. Now, we define $w_{\mathtt q_{1},\mathtt q_{2}}^{i}:=\mathtt q_{1}\left(\Gamma_{i}\right)-\mathtt q_{2}\left(\Gamma_{i}\right)$.
Since by our definition $\mathtt q_{1},\mathtt q_{2}$ are finitely supported
probability measures, \sloppy it follows $I_{0}:=\left\{ i\in I:\mathtt q_{1}\left(\Gamma_{i}\right)=\mathtt q_{2}\left(\Gamma_{i}\right)=0\right\} $ is cofinite. Moreover, $\sum_{i\in I}w_{\mathtt q_{1},\mathtt q_{2}}^{i}=\sum_{i\in I}\mathtt q_{1}\left(\Gamma_{i}\right)-\sum_{i\in I}\mathtt q_{2}\left(\Gamma_{i}\right)=1-1=0$.
Let $\mathtt r\in\FK{\mathbf{M}}B$. Then, for $\mathtt q_{1},\mathtt q_{2}\in\FK{\mathtt r}A$
\[
f\left(\mathtt q_{1}\right)-f\left(\mathtt q_{2}\right)=\sum_{i\in I}f\left(\Gamma_{i}\right)\cdot\left(\mathtt q_{1}\left(\Gamma_{i}\right)-\mathtt q_{2}\left(\Gamma_{i}\right)\right)=\sum_{i\in I}w_{\mathtt q_{1},\mathtt q_{2}}^{i}\cdot f\left(\Gamma_{i}\right).
\]
Now fix $\left\{ \Lambda_{j}:j\in J\right\} $ to be an enumeration
of all elements $\PK{\mathbf{M}}B$; without repetition. Since $\mathtt r$
is finitely supported probability measure, there is a finite $J_{\mathtt r}\subseteq J$
such that $\mathtt r\left(\Lambda_{j}\right)\neq0$ if and only if $j\in J_{\mathtt r}$.
For $\mathtt q_{1},\mathtt q_{2}\in\FK{\mathtt r}A$ it is clear that $w_{\mathtt q_{1},\mathtt q_{2}}^{i}=\mathtt q_{1}\left(\Gamma_{i}\right)=\mathtt q_{2}\left(\Gamma_{i}\right)=0$
when $\Gamma_{i}\left(A\right)\nsubseteq\Lambda_{j}\left(B\right)$,
for all $j\in J_{\mathtt r}$. Let $\mathtt q_{1}',\mathtt q_{2}'\in\FK BA$
such that $\mathtt q_{1}=\mathtt r\cdot\mathtt q_{1}'$ and $\mathtt q_{2}=\mathtt r\cdot\mathtt q_{2}'$.
Let $I_{\mathtt r}$ be the finite subset of $I$ such that $i\in I_{\mathtt r}$
if and only if $\Gamma_{i}\left(A\right)\subseteq\Lambda_{j}\left(B\right)$,
for some $j\in J_{\mathtt r}$. Therefore,
\[
f\left(\mathtt q_{1}\right)-f\left(\mathtt q_{2}\right)=\sum_{i\in I_{\mathtt r}}w_{\mathtt q_{1},\mathtt q_{2}}^{i}\cdot f\left(\Gamma_{i}\right)=\sum_{j\in J_{\mathtt r}}r_{j}\cdot\left(\sum_{i\in I_{\mathtt r};\Gamma_{i}\left(A\right)\subseteq\Lambda_{j}\left(B\right)}w_{\mathtt q_{1}',\mathtt q_{2}'}^{i}\cdot f\left(\Gamma_{i}\right)\right);
\]
where $r_{j}$'s are the coefficients calculated from $\mathtt r$: namely
$r_{j}=\mathtt r\left(\Lambda_{j}\right)$ for $j\in J_{\mathtt r}$. Note
that $\sum_{j\in J_{\mathtt r}}r_{j}=1$. We have already fixed an enumeration
of $\PK BA$. Then, it follows that
\[
f\left(\mathtt q_{1}\right)-f\left(\mathtt q_{2}\right)=\sum_{j\in J_{\mathtt r}}r_{j}\cdot\left(\sum_{1\leq i\leq m}w_{\mathtt q_{1}',\mathtt q_{2}'}^{i}\cdot f\left(\Lambda_{j}\circ \eta_{i}\right)\right).
\]
Denote $\bar{\mathrm f}_{j}:=\begin{pmatrix}f\left(\Lambda_{j}\circ \eta_{1}\right) &  & \cdots &  & f\left(\Lambda_{j}\circ \eta_{m}\right)\end{pmatrix}_{1\times m}$
for $j\in J_{\mathtt r}$ which we call it the \emph{coloring matrix}
of $\Lambda_{j}$. We can demonstrate all the calculations above in
the following matrix presentation:
\[
f\left(\mathtt q_{1}\right)-f\left(\mathtt q_{2}\right)=\begin{pmatrix}r_{j_{1}} &  & \cdots &  & r_{j_{t}}\end{pmatrix}_{1\times r}\times\begin{pmatrix}\bar{\mathrm f}_{1}\\
\\
\vdots\\
\\
\bar{\mathrm f}_{t}
\end{pmatrix}_{r\times m}\times\begin{pmatrix}w_{\mathtt q_{1}',\mathtt q_{2}'}^{1}\\
\\
\vdots\\
\\
w_{\mathtt q_{1}',\mathtt q_{2}'}^{m}
\end{pmatrix}_{m\times1};
\]
where $r:=\left|J_{\mathtt r}\right|$ and $J_{\mathtt r}=\left\{ j_{1},\cdots,j_{r}\right\} $.  We call the matrix $\begin{pmatrix}w_{\mathtt q_{1}',\mathtt q_{2}'}^{1}\\
\\
\vdots\\
\\
w_{\mathtt q_{1}',\mathtt q_{2}'}^{m}
\end{pmatrix}_{m\times1}$  a \emph{weight matrix}. Denote $\mathrm R^{\mathtt r}$ for the matrix $\begin{pmatrix}r_{j_{1}} &  & \cdots &  & r_{j_{r}}\end{pmatrix}_{1\times r}$
of coefficients of $\mathtt r$, which we call it a \emph{probability
matrix}. Denote $\mathrm F^{f}$ for the matrix $\begin{pmatrix}\bar{\mathrm f}_{1}\\
\\
\vdots\\
\\
\bar{\mathrm f}_{r}
\end{pmatrix}_{r\times m}$.
\begin{defn}
A $m\times1$-matrix $\mathrm W$ is called a\emph{ Dirac-weight} \emph{matrix}
if there are exactly one entry in $\mathrm W$ with value $1$, exactly one
entry value $-1$ and all the other entries of $\mathrm W$ are $0$.
\end{defn}
It is clear that there are at most $2\cdot\binom{m}{2}$ many different
Dirac-weight\emph{ }matrices.
\begin{lem}
A coloring function $f$ satisfies the convex $\leqslant$-Ramsey
condition if and only if there is a positive real valued probability
$1\times r$-matrix $\mathrm R$ such that $\mathrm R\times \mathrm F^{f}\times \mathrm W\leq\frac{1}{2}$,
for every Dirac-weight\emph{ }matrix $\mathrm W$.\end{lem}
\begin{proof}
($\Longrightarrow$) Obvious.\\
($\Longleftarrow$) Suppose $\mathrm V$ is a weight $m\times1$-matrix which
is not $0$ everywhere. Let $\mathrm V^{+}:=\left\{ v\in \mathrm V:v>0\right\} $
and $\mathrm V^{*}:=\left\{ v\in \mathrm V:v\notin \mathrm V^{+}\right\} $. In the light
of Corollary $\ref{cor:monoc}$, we only need to check the cases that
there are no monochromatic coloring matrices of $\leqslant$-closed
copies of $B$. One can show that $\sum_{v_{i}\in \mathrm V^{+}}f_{j,i}.v_{i}\leq\sum_{v_{i}\in \mathrm V^{+}}v_{i}\leq1$,
for all $1\leq j\leq r$. For each $1\leq j\leq r$ let $\mathrm W_{j}$ be
a Dirac-weight matrix such that
\begin{itemize}
\item $w_{l,1}=-1$ for some $1\leq l\leq m$ whenever $f_{j,l}=0$;
\item $w_{i,1}=1$ for some $1\leq i\leq m$ with $v_{i}\in \mathrm V^{+}$ and
$f_{j,i}=1$; In case there no such $v_{i}\in \mathrm V^{+}$ with $f_{j,i}=1$,
choose any $i\neq l$ and let $w_{i,1}=1$.
\end{itemize}
By the assumption, there is a real valued probability $1\times r$-matrix
$\mathrm R$ such that $\mathrm R\times \mathrm F^{f}\times \mathrm W_{j}\leq\frac{1}{2}$, for each
$1\leq j\leq r$. A straightforward calculation shows that $\mathrm R\times \mathrm F^{f}\times \mathrm V\leq\frac{1}{2}$.
\end{proof}

\subsection{The convex $\leqslant$-Ramsey property and its correspondence with
amenability}

\label{sub:3.2}

Similar to Section $\ref{sub:rams}$, we investigate the correspondence
between the amenability of the automorphism group of the generic structure
of a smooth class, and its convex $\leqslant$-Ramsey property. In
\cite{TatchMoore2011}, similar correspondence has been given for
the case of automorphism groups of Fra\"iss\'e-limit structures.
\begin{thm}
\label{thm:cors} Suppose $\mathbf{M}$ is the $\left(\mathcal{K},\leqslant\right)$-generic
structure of a smooth class $\left(\mathcal{K},\leqslant\right)$
with HP and AP. Then, the followings are equivalent:
\begin{enumerate}
\item $\Aut\left(\mathbf{M}\right)$ has the convex $\leqslant$-Ramsey
property\emph{ }with respect to $\left(\mathcal{K},\leqslant\right)$.
\item For every $A,B\in\mathcal{K}$ with $A\leqslant B$, there is $C\in\mathcal{K}$
such that $B\leqslant C$ and for every $f:\PK CA\rightarrow\left\{ 0,1\right\} $
there is $\mathtt p\in\FK CB$ such that for every $\mathtt q_1,\mathtt q_2\in\FK{\mathtt p}A$,
\[
\left|f\left(\mathtt q_1\right)-f\left(\mathtt q_2\right)\right|\leq\frac{1}{2}.
\]
\item For every $A,B\in\mathcal{K}$ with $A\leqslant B$ and every $\epsilon>0$,
there is $C\in\mathcal{K}$ such that $B\leqslant C$ and for every
$f:\PK CA\rightarrow\left[0,1\right]$ there is $\mathtt p\in\FK CB$
such that for every $\mathtt q_1,\mathtt q_2\in\FK{\mathtt p}A$,
\[
\left|f\left(\mathtt q_1\right)-f\left(\mathtt q_2\right)\right|\leq\epsilon.
\]
\item For every $A,B\in\mathcal{K}$ with $A\leqslant B$ and every $\epsilon>0$
and $n\in\mathbb{N}$, there is $C\in\mathcal{K}$ such that $B\leqslant C$
for every sequence of functions $f_{i}:\PK CA\rightarrow\left[0,1\right]$
with $i<n$, there is $\mathtt p\in\FK CB$ such that for every $\mathtt q_1,\mathtt q_2\in\FK{\mathtt p}A$
and $i<n$,
\[
\left|f_{i}\left(\mathtt q_1\right)-f_{i}\left(\mathtt q_2\right)\right|\leq\epsilon.
\]
\item $\Aut\left(\mathbf{M}\right)$ is amenable.
\end{enumerate}
\end{thm}
\begin{proof}
Proof of Theorem 6.1. in \cite{TatchMoore2011} can easily be modified
for this case.\end{proof}
\begin{rem}
We say a smooth have $\left(\mathcal{K},\leqslant\right)$ have the
\emph{convex $\leqslant$-Ramsey property} if condition (2) of the
theorem above holds for $\left(\mathcal{K},\leqslant\right)$. Then,
on the bases of Theorem $\ref{thm:cors}$, $\Aut\left(\mathbf{M}\right)$
has the convex $\leqslant$-Ramsey property\emph{ }with respect to
$\left(\mathcal{K},\leqslant\right)$ if and only if $\left(\mathcal{K},\leqslant\right)$
have the convex $\leqslant$-Ramsey property.
\end{rem}

\subsection{Main results}
\begin{defn}
Suppose $A,B\in\mathcal{K}$ and $A\leqslant B$. Let $f:\PK{\mathbf{M}}A\rightarrow\left\{ 0,1\right\} $
be a coloring function and $m:=\left|\PK BA\right|$. For $n>0$,
a $n\times m$-matrix $\mathrm Y$ is called a \emph{full-coloring matrix
of }$f$ if every row of the matrix $\mathrm Y$ corresponds to a coloring
matrix of a $\leqslant$-closed copy of $B$ in $\mathbf{M}$, and
conversely every coloring matrix of a $\leqslant$-closed copy of
$B$ corresponds to a row of $\mathrm Y$.
\end{defn}
\noun{${\bf Convention}$}. We always assume there is no repetition
of similar rows in $\mathrm Y$.
\begin{rem}
\label{rem} Let $\mathrm Y$ be the full-coloring matrix of a coloring function
$f$. Then	
\begin{enumerate}
\item It is an easy observation that when we interchange rows of $\mathrm Y$ we
obtain a full-coloring matrix of the same coloring function.
\item It follows from Lemma $\ref{lem:free-3}$, if the smooth class $\left(\mathcal{K},\leqslant\right)$
has the free-amalgamation property then, if $f$ is not constant,
$\mathrm Y$ should contain for every $w\in\left\{ 0,1\right\} $ and every
$1\leq i\leq m$ at least a row whose $i$-th entry is $w$, and moreover,
there are infinitely many distinct $\leqslant$-closed copies of $B$
such that the coloring matrix of them are exactly of that given row.
\item Suppose $\mathrm Y$ satisfies the convex $\leqslant$-Ramsey condition. This
means there are a finite set $\left\{ \Lambda_{i}:i\in I_{0}\right\} $
of $\leqslant$-closed embeddings of $B$ in $\mathbf{M}$, and a
finitely supported probability measure $\mathtt p\in\FK{\mathbf{M}}B$
such that $\mbox{supp}\left(\mathtt p\right)=\left\{ \Lambda_{i}:i\in I_{0}\right\} $
and $\left|f\left(\mathtt q_{1}\right)-f\left(\mathtt q_{2}\right)\right|\leq\frac{1}{2}$,
for every $\mathtt q_{1},\mathtt q_{2}\in\FK{\mathtt p}A$. Suppose $f'$ is
another coloring function that the set of rows of its full-coloring matrix contains the set of the coloring of the rows of $\left\{ \Lambda_{i}\left(B\right):i\in I_{0}\right\} $ under $f$. Then clearly the convex $\leqslant$-Ramsey condition also holds for $f'$.
\end{enumerate}
\end{rem}
So far, the matrices that we have considered are obtained from the
full-coloring matrices of coloring functions for some $A,B$ with
$A\leqslant B$ in a smooth class $\left(\mathcal{K},\leqslant\right)$.
The matrix presentation suggests the following definition, without
specifically referring to the class $\left(\mathcal{K},\leqslant\right)$
and any coloring function.
\begin{defn}
Let $\mathrm Y=\begin{bmatrix}\bar{\mathrm y}_{1}\\
\\
\vdots\\
\\
\bar{\mathrm y}_{n}
\end{bmatrix}$ be a $n\times m$-matrix whose entries are $0$ or $1$. Moreover,
assume there is no repetition of similar rows in $\mathrm Y$. We say $\mathrm Y$
satisfies \emph{the convex Ramsey condition} if there exist a $r\times m$-matrix
$\mathrm X=\begin{bmatrix}\bar{\mathrm x}_{1}\\
\\
\vdots\\
\\
\bar{\mathrm x}_{r}
\end{bmatrix}$ such that $\left\{ \bar{\mathrm x}_{1},\cdots,\bar{\mathrm x}_{r}\right\} \subseteq\left\{ \bar{\mathrm y}_{1},\cdots,\bar{\mathrm y}_{n}\right\} $
and a probability $1\times r$-matrix $\mathrm R$ such that for every Dirac-weight
$m\times1$-matrices $\mathrm W$
\[
\mathrm R\times \mathrm X\times \mathrm W\leq\frac{1}{2}.
\]
\end{defn}
\noun{Question. }It is an interesting question to fully understand
or classify, for fixed $n,m\in\mathbb{N}$, all $n\times m$-matrices
with the convex Ramsey condition. Then, the question of the convex
$\leqslant$-Ramsey property for a $\left(\mathcal{K},\leqslant\right)$
is reduced to investigate whether in a smooth class $\left(\mathcal{K},\leqslant\right)$
any of such matrices can be full-coloring matrix of a coloring function.
\begin{lem}
\label{claim-0-1} Suppose $\mathrm X$ is a $r\times m$-matrix with $r>0$
such that $\mathrm R$ contains a column whose entries are $1$, and a column
whose entries are $0$. Then, there is no probability $1\times r$-matrix
$\mathrm R$ such that
\[
\mathrm R\times \mathrm X\times \mathrm W\leq\frac{1}{2}
\]
 holds for all Dirac-weight matrices $\mathrm W$. \end{lem}
\begin{proof}
Suppose $\mathrm R$ is any probability matrix. Let $\mathrm Q:=\mathrm R\times \mathrm X$. Suppose
$\mathrm X_{i}^{c}$,$\mathrm X_{j}^{c}$ are two columns of $\mathrm X$ whose all its entries
are $1$, and $0$; respectively. Let $\mathrm W$ be a Dirac-weight matrix
such that $w_{i,1}:=1$, $w_{j,1}:=-1$ and $w_{k,1}=0$ for all $k\in\left\{ 1,\cdots,m\right\} \backslash\left\{ i,j\right\} $.
Then, $q_{1,i}=\mathrm R\times \mathrm X_{i}^{c}=\sum_{1\leq i\leq r}r_{i}=1$ and
$q_{1,j}=\mathrm R\times \mathrm X_{j}^{c}=0$. Now $\mathrm Q\times \mathrm W=q_{1,i}\cdot w_{i,1}+q_{1,j}\cdot w_{j,1}=1>\frac{1}{2}$
which is a failure for above condition. \end{proof}
\begin{lem}
\label{lem:matrixRa} For the following matrix
\textup{
\[
\mathrm  Y:=\begin{pmatrix}1 & 1 & 1 & 1 & 0 & 0\\
1 & 1 & 1 & 0 & 1 & 0\\
1 & 1 & 1 & 0 & 0 & 1\\
0 & 1 & 1 & 0 & 0 & 0\\
1 & 0 & 1 & 0 & 0 & 0\\
1 & 1 & 0 & 0 & 0 & 0
\end{pmatrix}_{6\times6},
\]
}
the convex Ramsey condition fails.\end{lem}
\begin{proof}
We need to prove for any $r>0$ and $r\times6$-matrix $\mathrm X$, whose
rows are chosen from rows of $\mathrm Y$, there is no probability $1\times r$-matrix
$\mathrm R$ such that $\mathrm R\times \mathrm X\times \mathrm W\leq\frac{1}{2}$, for all Dirac-weight
matrices $\mathrm W$. \\
From Lemma $\ref{claim-0-1}$, it follows that for $r=1,2$ there
is no probability $1\times r$-matrix $\mathrm R$ such that the convex Ramsey
condition holds for $\mathrm X$, since a column of with constant $1$ and
a column of constant $0$ appear in $\mathrm X$. For $r=3$, the only cases
that remain to be checked, again using Lemma $\ref{claim-0-1}$, are
\[
\begin{pmatrix}1 & 1 & 1 & 1 & 0 & 0\\
1 & 1 & 1 & 0 & 1 & 0\\
1 & 1 & 1 & 0 & 0 & 1
\end{pmatrix}\hspace{1em}\mbox{and}\hspace{1em}\begin{pmatrix}0 & 1 & 1 & 0 & 0 & 0\\
1 & 0 & 1 & 0 & 0 & 0\\
1 & 1 & 0 & 0 & 0 & 0
\end{pmatrix}.
\]
Let $\mathrm R:=\begin{pmatrix}r_{1} & r_{2} & r_{3}\end{pmatrix}_{1\times3}$
be any probability matrix i.e. $\sum_{1\leq i\leq3}r_{i}=1$. Then
$\mathrm Q:=\mathrm R\times \mathrm X=\begin{pmatrix}1 & 1 & 1 & r_{1} & r_{2} & r_{3}\end{pmatrix}_{1\times6}\hspace{1em}\mbox{and}\hspace{1em}\begin{pmatrix}r_{2}+r_{3} & r_{1}+r_{3} & r_{1}+r_{2} & 0 & 0 & 0\end{pmatrix}_{1\times6}$;
respectively. Now consider
\[
\mathrm W_{1}:=\begin{pmatrix}1\\
0\\
0\\
-1\\
0\\
0
\end{pmatrix},\hspace{1em}\hspace{1em}\mathrm W_{2}:=\begin{pmatrix}0\\
1\\
0\\
0\\
-1\\
0
\end{pmatrix}\hspace{1em}\mbox{and}\hspace{1em}\mathrm W_{3}:=\begin{pmatrix}0\\
0\\
1\\
0\\
0\\
-1
\end{pmatrix};
\]
to be Dirac-weight matrices. It is easy to see that $\mathrm Q\times \mathrm W_{1}=1-r_{1}$,
$\mathrm Q\times \mathrm W_{2}=1-r_{2}$, and $\mathrm Q\times \mathrm W_{3}=1-r_{3}$. The convex
Ramsey condition requires that $1-r_{1}\leq\frac{1}{2}$, $1-r_{2}\leq\frac{1}{2}$
and $1-r_{3}\leq\frac{1}{2}$. Hence $r_{i}\geq\frac{1}{2}$ for all
$i\in\left\{ 1,2,3\right\} $, which is contradictory with the fact
that $r_{1}+r_{2}+r_{3}\leq1$.\\
Now we assume $r>3$. Note that for any $r\times m$-matrix $\mathrm X$,
we have the following easy properties which will help us for the subsequent
calculations:
\begin{enumerate}
\item The number of entries $1$ in the first three columns is at most $3\cdot r$
and at least $2\cdot r$;
\item The number of entries $0$ in the first three columns is at most $r$.
Moreover, entry $0$ might not occurs in these columns;
\item Dually, the same statement as (2) also holds for the number of entries
$1$ in the last three columns.
\item In the first three columns never two entries $0$ appear in the same
row;
\item Dually, in the last three columns never two entries $1$ appear in
the same row.
\end{enumerate}
Suppose now $\mathrm R$ is a probability $1\times r$-matrix. Then, the
convex Ramsey condition requires the following matrix inequality to
be true:
\[
\mathrm Q\times\begin{pmatrix}1 & 1 & 1 & 0 & 0 & 0 & 0 & 0 & 0\\
0 & 0 & 0 & 1 & 1 & 1 & 0 & 0 & 0\\
0 & 0 & 0 & 0 & 0 & 0 & 1 & 1 & 1\\
-1 & 0 & 0 & -1 & 0 & 0 & -1 & 0 & 0\\
0 & -1 & 0 & 0 & -1 & 0 & 0 & -1 & 0\\
0 & 0 & -1 & 0 & 0 & -1 & 0 & 0 & -1
\end{pmatrix}_{6\times9}\leq\begin{pmatrix}\frac{1}{2} &  & \cdots &  & \frac{1}{2}\end{pmatrix}_{1\times9},
\]
where $\mathrm Q=\mathrm R\times \mathrm X$. Define the following set of indices: For $1\leq l\leq3$, let $I_{0}^{l}:=\left\{ i:r_{i,l}=0\right\} $
and for $3<l'\leq6$, let $I_{1}^{l'}:=\left\{ i:r_{i,l'}=1\right\} .$
By the above observation all $I_{0}^{l}$'s and $I_{1}^{l}$'s are
disjoint sets. This leads to the following inequalities:
\[
\begin{cases}
A_{1}+B & \geq\frac{1}{2}\\
A_{2}+B & \geq\frac{1}{2}\\
A_{3}+B & \geq\frac{1}{2}\\
A_{1}+C & \geq\frac{1}{2}\\
A_{2}+C & \geq\frac{1}{2}\\
A_{3}+C & \geq\frac{1}{2}\\
A_{1}+D & \geq\frac{1}{2}\\
A_{2}+D & \geq\frac{1}{2}\\
A_{3}+D & \geq\frac{1}{2}
\end{cases},
\]
 where $B:=\sum_{i\in I_{0}^{1}}r_{i}$, $C:=\sum_{i\in I_{0}^{2}}r_{i}$,
$D:=\sum_{i\in I_{0}^{3}}r_{i}$ and $A_{i}:=\sum_{j\in I_{1}^{i}}r_{j}$
for $1\leq i\leq3$. Then, it follows that $3\cdot\left(\sum_{1\leq i\leq3}A_{i}+B+C+D\right)\geq9\cdot\frac{1}{2}$
and subsequently $\sum_{1\leq i\leq3}A_{i}+B+C+D\geq\frac{3}{2}$.
However, since we have assumed $\mathrm R$ is probability matrix and the
set of indices of the above summations are disjoint, it follows $\sum_{1\leq i\leq3}A_{i}+B+C+D\leq1$;
a contradiction. Hence, the convex Ramsey condition fails for $\mathrm X$.\end{proof}
\begin{defn}
Let $\left(\mathcal{K},\leqslant\right)$ be a smooth class and $\mathbf M$
the $\left(\mathcal{K},\leqslant\right)$-generic model. Let $A,B\in\mathcal{K}$
with $A\leqslant B$ and assume $\Lambda_{i},\Lambda_{j}\in\PK MB$.
\begin{enumerate}
\item We say $\Lambda_{i}$ and $\Lambda_{j}$ are in the same \emph{connected
competent with respect to $A$} if there are $n\geq1$ and $\Lambda_{i_{1}},\cdots,\Lambda_{i_{n}}\in\PK {\mathbf M}B$
such that $\Lambda_{i_{1}}=\Lambda_{i}$, $\Lambda_{i_{n}}=\Lambda_{j}$
and $\Lambda_{i_{j}}\left(B\right)\cap\Lambda_{i_{j+1}}\left(B\right)$
contains at least one $\leqslant$-closed copy of $A$ for $1\leq j\leq n$. We say $\Lambda_{i}$ and $\Lambda_{j}$ have \emph{distance} $n-1$ if $n$ is the minimum number that satisfies the condition.
\item Let $m\geq2$. We say $\Lambda_{i}$ and $\Lambda_{j}$ lay on an
\emph{$m$-cycle of embeddings over $A$} if there exits distinct
$\Lambda_{i_{1}},\cdots,\Lambda_{i_{m}}\in\PK {\mathbf M}B$ where $\Lambda_{i_{1}}=\Lambda_{i}$,
$\Lambda_{i_{m}}=\Lambda_{j}$ such that:
\begin{enumerate}
  \item $\Lambda_{i_{1}}\left(B\right)\cap\Lambda_{i_{m}}\left(B\right)$
and $\Lambda_{i_{j}}\left(B\right)\cap\Lambda_{i_{j+1}}\left(B\right)$
contain at least one $\leqslant$-closed copy of $A$, for $1\leq j\leq m$;
  \item  No pairs of elements of the set $\left\{ \Lambda_{i_{1}}\left(B\right)\cap\Lambda_{i_{m}}\left(B\right), \Lambda_{i_{j}}\left(B\right)\cap\Lambda_{i_{j+1}}\left(B\right): 1\leq j\leq m \right\}$ contain a common $\leqslant$-closed copy  of $A$.
\end{enumerate}
\item We say $\left(A;B\right)$ is a \emph{tree-pair} if the following
conditions hold:
\begin{enumerate}
\item $\Lambda\left(B\right)\cap\Lambda^{\prime}\left(B\right)$ contains
at most one $\leqslant$-closed copy of $A$ for any two distinct
$\Lambda,\Lambda^{\prime}\in\PK {\mathbf M}B$;
\item Any two distinct $\Lambda,\Lambda'\in\PK {\mathbf M}B$ never lay in an $l$-cycle
for $l\geq2$.\end{enumerate}
\end{enumerate}
\end{defn}
\begin{thm}
\label{thm:main1} Suppose $\left(\mathcal{K},\leqslant\right)$ is
a smooth class with AP and HP, and $\mathbf{M}$ the $\left(\mathcal{K},\leqslant\right)$-generic
structure. Suppose there are $A,B\in\mathcal{K}$ and $A\leqslant B$
such that $\left(A;B\right)$ is a tree-pair with $\left|\PK BA\right|=6$.
Then, $\Aut\left(\mathbf{M}\right)$ does not have the convex $\leqslant$-Ramsey
property with respect to $\left(\mathcal{K},\leqslant\right)$.\end{thm}
\begin{proof}
Our strategy is to present a coloring function $f:\PK{\mathbf{M}}A\rightarrow\left\{ 0,1\right\} $
such that the full-coloring matrix of $f$ for copies of $B$ is
\[
\mathrm Y:=\begin{pmatrix}1 & 1 & 1 & 1 & 0 & 0\\
1 & 1 & 1 & 0 & 1 & 0\\
1 & 1 & 1 & 0 & 0 & 1\\
0 & 1 & 1 & 0 & 0 & 0\\
1 & 0 & 1 & 0 & 0 & 0\\
1 & 1 & 0 & 0 & 0 & 0
\end{pmatrix}_{6\times6}.
\]
Then, Lemma $\text{\ref{lem:matrixRa}}$ implies that $\Aut\left(\mathbf{M}\right)$
does not have the convex $\leqslant$-Ramsey property with respect
to $\left(\mathcal{K},\leqslant\right)$. \\
Take the connected components of elements of $\PK{\mathbf{M}}B$
with respect to $A$. It is enough
to present a suitable coloring function for each connected component, denoted by
$\left[\Lambda\right]$ where $\Lambda \in \PK{\mathbf{M}}B$. Since $\left(A;B\right)$
is a tree-pair, each connected component $\left[\Lambda\right]$ does not contain a cycle.  We can define a coloring function
$f$ on each $\left[\Lambda\right]$ by induction
on the distance of elements of $\left[\Lambda\right]$ from a fixed element in the class. \\
Choose a row $\bar{\mathrm y}$ in $\mathrm Y$, and let $f_{0}$ be a coloring of
$\leqslant$-closed copies of $A$ in $\Lambda\left(B\right)$ that
$f_{0}\left(\Lambda\right)=\bar{\mathrm y}$. For each $i\geq1$ and embeddings
of distance $i$ from $\Lambda$, we inductively extend $f_{i}$ to $f_{i+1}$ and then define $f:=\bigcup_{i<\omega}f_{i}$.
Here, we only explain how to define $f_{1}$ on elements of $\PK{\mathbf{M}}B$
which have distance one to $\Lambda$. Let $\Lambda^{\prime}\in\left[\Lambda\right]$ be such an element. Then $\Lambda'\left(B\right)\cap\Lambda\left(B\right)=\Lambda'\circ \eta_{i}\left(A\right)$
for \emph{exactly} one $i\in\left\{ 1,\cdots,6\right\}$ (follows from the definition of a tree-piar). \\
Choose a row $\bar{\mathrm y}$
of $\mathrm Y$ such that $y_{1, i}=f_{0}\left(\Lambda'\circ \eta_{i}\left(A\right)\right)$.
Let $f_{1}$ be such that $f_{1}\left(\Lambda'\right)=\bar{\mathrm y}$. Note that any two embeddings of distance one from $\Lambda$ intersect at most in one copy of $A$ in $\Lambda(B)$. This follows from the fact that there are no $2$-cycles in a connected competent of a tree-pair. Hence, one can assign a coloring, consistently, to each embedding of distance one and by Remark $\ref{rem:28}$, each row of $\mathrm Y$, will eventually appear in $f_{1}$. A similar argument works for embeddings of distance $i\geq 1$ inductively, as there are no $l$-cycles in a connected component in a tree-pair for $2\leq l\leq 2i$. Therefore, we have $f_i$'s defined and finally the desired $f$ is obtained for each connected component.
\end{proof}

\section{Amenability of automorphism groups of Ab-initio generic structures}

\label{sec:4}

Using the correspondence in Subsection $\ref{sub:3.2}$ and Theorem
$\ref{thm:cors}$, we show that the automorphism groups of Hrushovski-Fra\"iss\'e
structures which are obtained from pre-dimensions with rational coefficients
are not amenable groups. Moreover, for Hrushovski structures in these
cases, we strengthen Theorem $\ref{thm:extamngraph}$, by showing that
the automorphism groups of their ordered generic structures are not
extremely amenable (Theorem $\ref{thm:extamngraph}$ only solves the binary case).

\subsection{Pre-dimension functions with rational coefficients.}

\subsubsection{Setting}

Suppose $\left(\mathcal{K}_{\alpha}^{+},\leqslant\right)$ is an ab-initio
class of finite relational $\mathfrak{L}$-structures, where $\mathfrak{L}=\left\{ \mathfrak{R}\right\} $,
that has been defined in Subsection $\ref{sub:ab}$. In this section,
we assume $\alpha$ is a rational number. For simplicity, we denote
$\delta$,$\mathcal{K}_{0}$,$\leqslant$ and $\mathbf{M}_{0}$ for
$\delta_{\alpha}$,$\mathcal{K}_{\alpha}^{+}$,$\leqslant_{\alpha}$
and $\mathbf{M}_{\alpha}$; respectively. We fix $\alpha=2$; the
arguments can be easily modified for any rational $\alpha\geq1$. Recall that we denote $AB$ for the $\mathfrak L$-substructure $A\cup B$ in $C$ and $\delta(A/B)=\delta (AB)-\delta (A)$ where $A,B,C\in \mathcal K_0$ and $A,B\subseteq C$.
\begin{fact}
(See \cite{Wag1}) For each finite $A,B,C\in\mathbf{M}_{0}$, the following holds
\[
\delta\left(ABC\right)=\delta\left(AB/C\right)+\delta\left(C\right)=\delta\left(A/BC\right)+\delta\left(B/C\right)+\delta\left(C\right).
\]
\end{fact}
Recall the definition of $0$-minimally
algebraic sets from Subsection $\ref{sub:ab}$ (i.e. Definition \ref{zero-min}). The following remark is used later in the proof of Theorem $\ref{thm:ab-}$.
\begin{rem}\label{rem-new-44}
 Suppose $A\in \mathcal K_0$, and let $m\geq \mbox{max}\lbrace3, |A|\rbrace$ be an integer.  Consider $C_m=\left\lbrace c_1,\cdots,c_m\right\rbrace$ such that $\bigwedge_{1\leq i<m}\mathfrak {R}^{C_m}\left(c_i,c_{i+1} \right)\wedge \mathfrak R^{C_m}(c_1,c_m)$; i.e. $C_m$ is a single $m$-cycle. Note that $C_m\in \mathcal K_0$.  Suppose $(a_i:1\leq i\leq n)$ is an enumeration of elements of $A$. Now let $D_m:=A\dot\cup C_m$ such that $\bigwedge_{1\leq i\leq n} \mathfrak R^{D_m} (c_i,a_i)\wedge \bigwedge_{n<i\leq m} \mathfrak R^{D_m}(c_i,a_n)$. It is easy to check that $\delta (D_m/A)=0$, $D_m\in \mathcal K_0$ and $C_m$ is $0$-minimally algebraic over $A$. Hence, there are infinitely many non-isomorphic $0$-minimally algebraic sets over each $A$ in $\mathcal K_0$.
\end{rem}

\begin{thm}
\label{thm:ab-} There are $A,B$ in $\mathcal{K}_{0}$ with $A\leqslant B$
and $\left|\PK BA\right|=6$ such that $\left(A;B\right)$ is a tree-pair.
Therefore, $\Aut\left(\mathbf{M}_{0}\right)$ does not have the convex
$\leqslant$-Ramsey property with respect to $\left(\mathcal{K}_{0},\leqslant\right)$.\end{thm}
\begin{proof}
Note that the number $6$ (i.e. the number of copies $A$ in $B$) is only needed to obtain the tree-pair that is required in Theorem \ref{thm:main1} and proving the existence of tree-pairs that have more that $6$ copies of $A$ are similar. Define $\mathrm{P}_{2}\left(6\right):=\left\{ u\subseteq\left\{ 1,\cdots,6\right\} :\left|u\right|=2\right\}$.

\begin{claim*}[A]
 There are $A,B\in \mathcal K_0$ such that following conditions hold
\begin{enumerate}
 \item $A\in\mathcal{K}_{0}$ such that $|A|>2$, $\delta\left(A\right)=3$ and there is no $A'\subseteq A$ with $\delta\left(A'\right)=0$.
 \item $B$ contains exactly six disjoint $\leqslant$-closed isomorphic copies $A_{1},\cdots,A_{6}$ of $A$.
 \item $B=\dot{\bigcup}_{1\leq i\leq6}A_{i}\dot{\cup}\dot{\bigcup}_{u\in\mathrm{P}_{2}\left(6\right)}X_{u}$ such that the followings hold:
\begin{enumerate}
\item $\delta\left(X_u /\left(A_{u_1}A_{u_2}\right)\right)=-1$ for each $u:=\{u_1,u_2\}\in \mathrm{P}_{2}(6)$;
\item $\ensuremath{X_{u}A_{u_{1}}A_{u_{2}}\ncong X_{v}A_{v_{1}}A_{v_{2}}}$ for all $v,u\in \mathrm{P}_{2}(6)$ where $u\neq v$;
\item $\delta\left(X'/\left(A_{u_1}A_{u_2}\right)\right)\geq0$ for all $\emptyset\neq X'\subsetneqq X_u$ where $u\in \mathrm{P}_{2}(6)$.
\end{enumerate}
\item $\delta (B)=\delta (A)=3$.
\end{enumerate}
\end{claim*}
\begin{proof}[Proof of Claim A]
 For each $1\leq i\leq 6$ let $A_{i}:=\left\{ a_{1}^{i},a_{2}^{i},a_{3}^{i}\right\}$ be an $\mathfrak{L}$-structure with $\mathfrak{R}^{A_{i}}\left(a_{1}^{i},a_{2}^{i}\right)\wedge\mathfrak{R}^{A_{i}}\left(a_{2}^{i},a_{3}^{i}\right)\wedge\mathfrak{R}^{A_{i}}\left(a_{3}^{i},a_{1}^{i}\right)$ such that $A_{i}\cap A_{j}=\emptyset$ for $1\leq i\neq j\leq6$. Let $A$ denote the isomorphic type of $A_{i}$'s. It is clear that $\delta\left(A\right)=3$ and $A\in\mathcal{K}_{0}$. Fix $\zeta:\mathrm{P}_{2}\left(6\right)\rightarrow\left\{ 1,\cdots,15\right\}$ to  be an enumeration of elements of $\mathrm{P}_{2}\left(6\right)$, without repetition. For each $u\in\mathrm{P}_{2}\left(6\right)$ put $\mathfrak{m}_{u}:=6\cdot\zeta\left(u\right)$. Now let $X_{u}$ be isomorphic to $C_{\mathfrak{m}_{u}}:=\left\{ c_{1},\cdots,c_{\mathfrak{m}_{u}}\right\}$ (a cycle of length $\mathfrak{m}_{u}$; see Remark \ref{rem-new-44}). It is clear that $\delta\left(C_{\mathfrak{m}_{u}}\right)=\mathfrak{m}_{u}$. As $\mathfrak{m}_{u}\geq6$ then $X_{u}$
  does not contain any substructure isomorphic to $A$, for each $u\in\mathrm{P}_{2}\left(6\right)$. Now let $B=\dot{\bigcup}_{1\leq i\leq6}A_{i}\dot{\cup}\dot{\bigcup}_{u\in\mathrm{P}_{2}\left(6\right)}X_{u}$ be an $\mathfrak{L}$-structure such that $A_{i}$'s and $X_{u}$ 's are $\mathfrak{L}$-structures as above, for $1\leq i\leq6$ and $u\in\mathrm{P}_{2}\left(6\right)$; respectively, with the following additional relations: For each $u=\left\{ u_1,u_2\right\} \in\mathrm{P}_{2}\left(6\right)$
  we have $\mathfrak{R}^{B}\left(c_{1},a_{2}^{u_2}\right)\mathfrak{\bigwedge}_{1\leq p\leq\frac{\mathfrak{m}_{u}}{2}}\mathfrak{R}^{B}\left(c_{2p-1},a_{p^{*}}^{u_1}\right)\wedge\mathfrak{R}^{B}\left(c_{2p},a_{p^{*}}^{u_2}\right)$  where $p^{*}\in\left\{ 1,2,3\right\}$ such that $p^{*}\equiv p\left(\mbox{mod }3\right)$. One can check that in the $\mathfrak{L}$-structure $B$ we have $\delta\left(X_{u}/A_{u_1}A_{u_2}\right)=\mathfrak{m}_{u}-\left(\mathfrak{m}_{u}+1\right)=-1$
  where $u=\left\{ u_1,u_2\right\} \in\mathrm{P}_{2}\left(6\right)$. Moreover, the only isomorphic copies of $A$  are $A_{i}$'s for $1\leq i\leq6$ in $B$. Hence $\left|\PK BA\right|=6$. Furthermore as the $X_u$ contain cycles of different length, hence 3-(b) follows. One can also check that 3-(c) also follows. Finally, one can see  $\delta\left(B\right)=6\cdot3-\binom{6}{2}=18-15=3$, and $B\in\mathcal{K}_{0}$.
\end{proof}

Let $A, B\in \mathcal K_0$ be $\mathfrak L$-structures that are obtained from Claim (A). Then
\begin{claim*}[B]
\label{lem:big-intersc}
\begin{enumerate}
\item $\delta\left(B'\right)\geq3$ for every $\bigcup_{\leq i\leq6}A_{i}\subseteq B'\subseteq B$.
\item Suppose $B'\subsetneqq B$ such that it contains at least two $\leqslant$-closed
copy of $A$. Then, $\delta\left(B'\right)>3$ and $\cl{B'}=B$.
\item $A_{i}\leqslant B$, for each $i\in\left\{ 1,\cdots,6\right\} $.
\end{enumerate}
\end{claim*}
\begin{proof}[Proof of Claim B]
(1) By Condition 3-(c) of Claim (A), we have  $\delta\left(X'/\left(A_{u_1}A_{u_2}\right)\right)\geq0$,
for all $\emptyset\neq X'\subsetneqq X_{u}$. Therefore,
$\delta\left(B'\right)\geq\delta\left(B\right)=3$.

(2) If $B'$ contains all $A_{i}$'s for $1\leq i\leq6$, then at
least one $X_{u}$ does not fully contain is $B'$ where $u\in \mathrm P_2(6)$.
So it is enough to show if $X'\subsetneqq X_{u }$,
then $\delta\left(\left(B\backslash X_{u }\right)\cup X'\right)>3$.
We have the following
\[
\begin{array}{ccc}
\delta\left(B\right) & = & \delta\left(B\backslash X_{u }\right)+\delta\left(X_{u }/\left(B\backslash X_{u }\right)\right)\\
 & = & \delta\left(B\backslash X_{u }\right)+\delta\left(X_{u }/A_{u_1}A_{u_2}\right)\\
 & = & \delta\left(B\backslash X_{u }\right)-1.
\end{array}
\]
Therefore, $\delta\left(B\backslash X_{u }\right)=\delta\left(B\right)+1>3$.
Again by the assumption $\delta\left(X'/A_{u_1}A_{u_2}\right)\geq0$ and
then
\[
\begin{array}{ccc}
\delta\left(\left(B\backslash X_{u }\right)\cup X'\right) & = & \delta\left(B\backslash X_{u }\right)+\delta\left(X'/\left(B\backslash X_{u }\right)\right)\\
 & = & \delta\left(B\backslash X_{u }\right)+\delta\left(X'/\left(A_{u_1}A_{u_2}\right)\right)\\
 & > & \delta\left(B\backslash X_{u }\right)+\delta\left(X_{u }/A_{u_1}A_{u_2}\right)\\
 & > & 3.
\end{array}
\]
Suppose \textbf{$B'$} contains only $n$-many $\leqslant$-closed
copies of $A$ for $1<n<6$. Without loss of generality assume $A_{1},\cdots,A_{n}\subseteq B'$.
With abuse of notation by $i\in n$, we mean $1\leq i\leq n$. Then,
$B'=\bigcup_{i\in n}A_{i}\dot{\cup}X^{*}\dot{\cup}A^{*}$ where $X^{*}\subseteq\dot{\bigcup}_{u\in \mathrm P_2(6) }X_{u }$
and $A^{*}\subseteq\bigcup_{j\notin n}A_{j}$. Note that by our assumption
$A^{*}$ is a disjoint union of proper subsets of $A_{j}$. Let $X^{*}:=X_{1}^{*}\dot{\cup}X_{2}^{*}$
where $X_{1}^{*}$ is the union of all $X_{u }$'s
such that $X^{*}\cap X_{u }=X_{u }$
and let $X_{2}^{*}:=X^{*}\backslash X_{1}^{*}$. Note that $\delta\left(X'/\bigcup_{1\leq i\leq6}A_{i}\right)=\delta\left(X'/A_{u_1}A_{u_2}\right)$
for $X'\subseteq X_{u }$. Then, by putting $X:=\dot{\bigcup}_{u\in \mathrm P_2(6)}X_{u }$,
from Condition 3-(c) of Claim (A), it follows
\[
\delta\left(X^{*}/\bigcup_{i\in n}A_{i}\dot{\cup}A^{*}\right)=\delta\left(X_{1}^{*}/\bigcup_{i\in n}A_{i}\dot{\cup}A^{*}\right)+\delta\left(X_{2}^{*}/\bigcup_{i\in n}A_{i}\dot{\cup}A^{*}\right)\geq\delta\left(X/\bigcup_{i\in n}A_{i}\dot{\cup}A^{*}\right).
\]
Thus
\[
\begin{array}{ccc}
\delta\left(B'\right) & = & \delta\left(\bigcup_{i\in n}A_{i}\dot{\cup}X^{*}\dot{\cup}A^{*}\right)\\
 & = & \delta\left(\bigcup_{i\in n}A_{i}\dot{\cup}A^{*}\right)+\delta\left(X^{*}/\bigcup_{i\in n}A_{i}\dot{\cup}A^{*}\right)\\
 & \geq & \delta\left(\bigcup_{i\in n}A_{i}\dot{\cup}A^{*}\right)+\delta\left(X/\bigcup_{i\in n}A_{i}\dot{\cup}A^{*}\right)\\
 & \geq & \delta\left(B'X\right).
\end{array}
\]
Now we only need to calculate $\delta\left(B'X\right)$. It is easy
to see that $\delta\left(\bigcup_{i\in n}A_{i}\cup\bigcup_{u\in \mathrm P_2(6)}X_{u }\right)=n\cdot3-\binom{n}{2}>3$.
Let $X_{1}^{+}:=\bigcup_{u\in \mathrm P_2(6)}X_{u}$ and $X_{2}^{+}:=X\backslash X_{1}^{+}$.
\sloppy We have  $\delta\left(X_{u }/\bigcup_{i\in n} A_{i}\dot{\cup}A^{*}\right)>\delta\left(X_{u }/A_{u_1}A_{u_1}\right)$
when $X_{u }\subseteq X_{2}^{+}$. Therefore, it
follows that $\delta\left(X_{2}^{+}/\bigcup_{i\in n}A_{i}\dot{\cup}A^{*}\right)\geq0$.
Now
\[
\begin{array}{ccc}
\delta\left(B'X\right) & = & \delta\left(\bigcup_{i\in n}A_{i}\dot{\cup}X\dot{\cup}A^{*}\right)\\
 & = & \delta\left(\bigcup_{i\in n}A_{i}\dot{\cup}A^{*}\right)+\delta\left(X/\bigcup_{i\in n}A_{i}\dot{\cup}A^{*}\right)\\
 & = & \delta\left(\bigcup_{i\in n}A_{i}\right)+\delta\left(A^{*}\right)+\delta\left(X/\bigcup_{i\in n}A_{i}\dot{\cup}A^{*}\right)\\
 & = & \delta\left(\bigcup_{i\in n}A_{i}\right)+\delta\left(A^{*}\right)+\delta\left(X_{1}^{+}/\bigcup_{i\in n}A_{i}\dot{\cup}A^{*}\right)+\delta\left(X_{2}^{+}/\bigcup_{i\in n}A_{i}\dot{\cup}A^{*}\right)\\
 & \geq & \delta\left(\bigcup_{i\in n}A_{i}\right)+\delta\left(A^{*}\right)+\delta\left(X_{1}^{+}/\bigcup_{i\in n}A_{i}\dot{\cup}A^{*}\right)\\
 & \geq & \delta\left(\bigcup_{i\in n}A_{n}\cup X_{1}^{+}\right)+\delta\left(A^{*}\right)\\
 & \geq & \delta\left(\bigcup_{i\in n}A_{n}\cup X_{1}^{+}\right)\\
 & > & 3.
\end{array}
\]
Hence, $\cl{B'}=B$.

(3) Follows from (2) and (1).\end{proof}
\begin{claim*}[C]
\label{lem:3-inter} Let $\Lambda_{1},\Lambda_{2},\Lambda_{3}\in\PK{\mathbf{M}_{0}}B$.
Then, the followings hold
\begin{enumerate}
\item If $\Lambda_{1}\left(B\right)\cap\Lambda_{2}\left(B\right)$ contains
at least two $\leqslant$-closed copy of $A$, then $\Lambda_{1}=\Lambda_{2}$.
\item Suppose $\Lambda_{1}\left(B\right)\cap\Lambda_{2}\left(B\right)$
contains exactly one $\leqslant$-closed copy of $A$. Then, $\Lambda_{1}\left(B\right)\Lambda_{2}\left(B\right)=\Lambda_{1}\left(B\right)\otimes_{\Lambda_{1}\left(B\right)\cap\Lambda_{2}\left(B\right)}\Lambda_{2}\left(B\right)$
and $\delta\left(\Lambda_{1}\left(B\right)\Lambda_{2}\left(B\right)\right)=3$.
Furthermore, $\delta\left(\Lambda_{1}\left(B\right)\cap\Lambda_{2}\left(B\right)\right)=3$.
\end{enumerate}
\end{claim*}
\begin{proof}[Proof of Claim C]
(1) Follows from Claim B.

(2) For simplicity let $B_{1}:=\Lambda_{1}\left(B\right)$ and $B_{2}:=\Lambda_{2}\left(B\right)$.

It is easy to verify that
\[
\delta\left(B_{1}B_{2}\right)\leq\delta\left(B_{1}\right)+\delta\left(B_{2}\right)-\delta\left(B_{1}\cap B_{2}\right).
\]
The equality holds if and only if $B_{1}$ and $B_{2}$ are in the
free-amalgamation over $B_{1}\cap B_{2}$.\\
Now
\[
3\leq\delta\left(B_{1}B_{2}\right)\leq3+3-3.
\]
Hence, $B_{1}$ and $B_{2}$ are in the free-amalgamation over $B_{1}\cap B_{2}$
and $B_{1}B_{2}$ is $\leqslant$-closed. Moreover, $\delta\left(\Lambda_{1}\left(B\right)\cap\Lambda_{2}\left(B\right)\right)=3$
follows from the fact that $\Lambda_{1}\left(B\right)\cap\Lambda_{2}\left(B\right)$
is a $\leqslant$-closed set that it contains a $\leqslant$-closed
copy of $A$, and it contained in $\Lambda_{1}\left(B\right)$.
\end{proof}
Now, we want to show $\left(A;B\right)$ is a tree-pair. Suppose,
on the contrary, that a connected component of $\PK{\mathbf M_0}B$
does contain an $m$-cycle for some $m>1$ (i.e. there are distinct
$\Lambda_{1},\cdots,\Lambda{}_{m}\in\PK{\mathbf{M}_{0}}B$ such that
$\Lambda_{j},\Lambda_{j+1}$ have distance one for $1\leq j<m$,
and $\Lambda_{1}\left(B\right)\cap\Lambda_{m}\left(B\right)$ contains at least one $\leqslant$-closed copy of $A$ such that $\Lambda_{1}\left(B\right)\cap\Lambda_{m}\left(B\right)\neq \Lambda_{1}\left(B\right)\cap\Lambda_{2}\left(B\right)$). For simplicity,
let $B_{j}:=\Lambda_{j}\left(B\right)$, for $1\leq j\leq m$. We
further assume that the $m$-cycle is minimal which implies that the
intersection of each successor pair of elements of $B_{j}$'s are
distinct. By Claim C, $B_{i}\cap B_{i+1}$
contains exactly one $\leqslant$-closed copy of $A$. Now
\[
\begin{array}{ccc}
\delta\left(B_{1}\cdots B_{m}\right) & \leq & \delta\left(B_{1}\cdots B_{m-1}\right)+\delta\left(B_{m}\right)-\delta\left(\left(B_{1}\cdots B_{m-1}\right)\cap B_{m}\right)\\
 & \leq & \delta\left(B_{1}\cdots B_{m-1}\right)+3-2\cdot\delta\left(A\right)\\
 & \leq & \delta\left(B_{1}\cdots B_{m-1}\right)-3\\
 & \leq & \delta\left(B_{1}\cdots B_{m-2}\right)+\delta\left(B_{m-1}\right)-\delta\left(\left(B_{1}\cdots B_{m-2}\right)\cap B_{m-1}\right)-3\\
 & \leq & \delta\left(B_{1}\cdots B_{m-2}\right)-3\\
 & \vdots & \vdots\\
 & \leq & \delta\left(B_{1}\right)-3=0.
\end{array}
\]
This contradicts with the fact that for each $B_{i}$ we have $B_{i}\subseteq B_{1}\cdots B_{n}$
and $B_{i}$ is $\leqslant$-closed in $\mathbf{M}_{0}$. Therefore,
each connected component of $\PK{\mathbf{M}_{0}}B$ with respect to $A$ does not contain a cycle
and hence $\left(A;B\right)$ is a tree-pair.\emph{ }

\end{proof}

Epand the language $L$ by adding a binary relation $\prec$ and let $\mathfrak{L}^{+}:=\mathfrak{L}\cup\left\{ \prec\right\} $. Let $\mathcal{K}_{0}^{+}$ be
the set all $\mathfrak{L^{+}}$-expansions of structures $C\in\mathcal{K}_{0}$,
in which the relation $\prec$ is interpreted as a linear-ordering
on the universe of $C$. For $E^{\mathfrak{L}^{+}},F^{\mathfrak{L}^{+}}\in\mathcal{K}_{\alpha,\mathfrak{L}}^{+}$
we define $E^{\mathfrak{L}^{+}}\leqslant^{+}F^{\mathfrak{L}^{+}}$
if and only if $E^{\mathfrak{L}}\subseteq F^{\mathfrak{L}}$ and $E\leqslant_{\alpha}F$
where $A^{\mathfrak{L}}$ and $B^{\mathfrak{L}}$ are $\mathfrak{L}$-expansions
of graphs $A$ and $B$; respectively. Similar to the proof above,
we can show $\left(A^{\mathfrak{L}^{+}};B^{\mathfrak{L}^{+}}\right)$
in again a tree-pair where the $A^{\mathfrak{L}^{+}}\subseteq B^{\mathfrak{L}^{+}}$,
$A^{\mathfrak{L}^{+}}\upharpoonright\mathfrak{L}=A$ and $B^{\mathfrak{L}^{+}}\upharpoonright\mathfrak{L}=B$.
Therefore, the following theorem is established. Indeed, in the case
of pre-dimension functions with rational coefficients, Theorem $\ref{thm:extamngraph}$
extends to a wider class of generic structures.
\begin{thm}
\label{cor:extrem-} The automorphism groups of ordered Hrushovski
generic structures that are obtained from pre-dimension functions
with rational coefficients are not extremely amenable.\end{thm}
\begin{rem}
In \cite{MacS}, it was asked whether there are any links between
the \emph{extension property} (known also as \emph{Hrushovski property}
in \cite{MR2308230,Ivan}) of the Fra\"iss\'e class and the extremely
amenability of the automorphism group of the Fra\"iss\'e-limit. By
a result of the first author in \cite{Zthesis}, the class $\mathcal{K}_{0}$
does not have the extension property. Existence of certain kind of
tree-pairs and the extension property of the class $\mathcal{K}_{0}$
seems to be related.
\end{rem}

\subsection{Collapsed ab-initio generic structures}

In \cite{Hrunew}, in order to obtain a strongly minimal structure, Hrushovski restricts the uncollapsed ab-initio class $\mathcal{K}_{\alpha}^{+}$
to a smaller class $\mathcal{K}_{\alpha}^{\mu}$, using a finite-to-one
function $\mu$ over the 0-minimally algebraic elements. We have already mentioned the class $\mathcal K_\alpha^\mu$ and the $\left(\mathcal K_\alpha^\mu,\leqslant_\alpha\right)$-generic structure $\mathbf{M}_\alpha^\mu$ in Subsection \ref{sub:ab}.  The structure $\mathbf{M}_\alpha^\mu$ is called the \emph{collapsed} ab-initio generic structure. It has to be indicated that the relation  $\mathfrak R$ in the language $\mathfrak L$ that Hrushovski considers is a ternary symmetric relation and $\alpha=1$. Note that
\textquotedblleft collapsing\textquotedblright{} with the $\mu$ function
is applicable just for ab-initio generic structures which are obtained
from pre-dimension functions with rational coefficients. Similar to Theorem \ref{thm:ab-}, one can show there are $A,B\in\mathcal{K}_{\alpha}^{\mu}$
such that $\left|\PK BA\right|=6$ and  $\left(A;B\right)$ form a tree-pair. Hence,

\begin{thm}
The automorphism groups of collapsed ab-initio generic structures and specially the
automorphism group of Hrushovski's strongly minimal set are not amenable.
\end{thm}

\subsection{Remaining cases}

In the previous subsection, we only dealt with ab-initio generic structures
that are obtained from pre-dimesnion functions with rational coefficients.
For ab-initio generic structures that are obtained from pre-dimension
functions with irrational coefficients, the amenability of their automorphism
group is left unanswered in this manuscript.  This includes the $\omega$-categorical
pseudo-plane constructed by Hrushovski (see \cite{Wag1}) and $\mathbf M_\alpha^{\mathsf f}$ that we have mentioned in Subsection \ref{sub:ab}. There are also other
interesting variants of Hrushovski's construction that are simple (see \cite{MR1982948}). In all these cases tree-pairs do not exist. However, it seems plausible
to modify the techniques of the present paper to assign a coloring
function, to a pair $\left(A;B\right)$ whose graph might have cycles,
that corresponds to a matrix which does not satisfy the convex Ramsey
condition. As we mentioned before,  David M. Evans shows, using a different method, the automorphism groups of generic structures that are obtained from pre-dimension functions with irrational coefficients and the $\omega$-categorical generic structures are not amenable.

\bibliographystyle{amsplain}
\bibliography{ALL}

\providecommand{\bysame}{\leavevmode\hbox to3em{\hrulefill}\thinspace}
\providecommand{\MR}{\relax\ifhmode\unskip\space\fi MR }
\providecommand{\MRhref}[2]{%
  \href{http://www.ams.org/mathscinet-getitem?mr=#1}{#2}
}
\providecommand{\href}[2]{#2}
\begin{thebibliography}{10}

\bibitem{Balshi}
John~T. Baldwin and Niandong Shi, \emph{Stable generic structures}, Ann. Pure
  Appl. Logic \textbf{79} (1996), no.~1, 1--35. \MR{1390325 (97c:03103)}

\bibitem{E:pre}
David~M. Evans, \emph{$\aleph_0$-categorical structures with a predimension},
  Annals of Pure and Applied Logic (2002), no.~116, 157--186.

\bibitem{Zthesis}
Zaniar {Ghadernezhad}, \emph{{Automorphism groups of generic structures.}},
  M\"unster: Univ. M\"unster, Mathematisch-Naturwissenschaftliche Fakult\"at,
  Fachbereich Mathematik und Informatik (Diss.), 2013 (English).

\bibitem{MR1462612}
Wilfrid Hodges, \emph{A shorter model theory}, Cambridge University Press,
  Cambridge, 1997. \MR{1462612 (98i:03041)}

\bibitem{Hrunew}
Ehud Hrushovski, \emph{A new strongly minimal set}, Ann. Pure Appl. Logic
  \textbf{62} (1993), no.~2, 147--166, Stability in model theory, III (Trento,
  1991). \MR{1226304 (94d:03064)}

\bibitem{Ivan}
A.~Ivanov, \emph{An $\omega$-categorical structure with amenable automorphism
  group}, ArXiv (2014).

\bibitem{KPT}
A.~Kechris, V.~Pestov, and S.~Todorcevic, \emph{Fra\"iss\' e limites, ramsey
  theory, and topological dynamics of automorphism groups}, GAFA (2005),
  no.~15, 106 --189.

\bibitem{MR2308230}
Alexander~S. Kechris and Christian Rosendal, \emph{Turbulence, amalgamation,
  and generic automorphisms of homogeneous structures}, Proc. Lond. Math. Soc.
  (3) \textbf{94} (2007), no.~2, 302--350. \MR{2308230 (2008a:03079)}

\bibitem{KueLas}
D.~W. Kueker and M.~C. Laskowski, \emph{On generic structures}, Notre Dame J.
  Formal Logic \textbf{33} (1992), no.~2, 175--183. \MR{1167973 (93k:03032)}

\bibitem{MacS}
Dugald Macpherson, \emph{A survey of homogeneous structures}, Discrete Math.
  \textbf{311} (2011), no.~15, 1599--1634. \MR{2800979 (2012e:03063)}

\bibitem{MR1982948}
Massoud Pourmahdian, \emph{Simple generic structures}, Ann. Pure Appl. Logic
  \textbf{121} (2003), no.~2-3, 227--260. \MR{1982948 (2004d:03079)}

\bibitem{Prom}
Hans~J\"urgen Pr\"omel, \emph{Ramsey theory for discrete structures}, Springer,
  2013.

\bibitem{TatchMoore2011}
Justin Tatch~Moore, \emph{Amenability and ramsey theory}, Fund. Math.
  \textbf{220} (2013), 263--280.

\bibitem{Wag1}
Frank~O. Wagner, \emph{Relational structures and dimensions}, Automorphisms of
  first-order structures, Oxford Sci. Publ., Oxford Univ. Press, New York,
  1994, pp.~153--180. \MR{1325473}

\end{thebibliography}

\end{document}